\documentclass[10pt]{article}

\usepackage{etex}
\usepackage[english]{babel}
\usepackage[utf8]{inputenc}
\usepackage[T1]{fontenc}
\usepackage{lmodern}

\usepackage{amsmath, latexsym, amssymb, amsfonts, amsthm}
\usepackage[colorlinks,citecolor=black,filecolor=black,linkcolor=black,urlcolor=black ]{hyperref}
\usepackage{graphicx}
\usepackage{subfigure} 

\def\ra {\rightarrow}

\def\be{\begin{equation}}   \def\ee{\end{equation}}
\def\ba   {\begin{array}}      \def\ea   {\end{array}}
\def\bea  {\begin{eqnarray}}   \def\eea  {\end{eqnarray}}
\def\bean {\begin{eqnarray*}}  \def\eean {\end{eqnarray*}}

\newtheorem{theorem} {Theorem}
\newtheorem{lemma}{Lemma}
\newtheorem{definition} {Definition}

\newtheorem{corollary} {Corollary}
\newtheorem{remark}{Remark}

\newtheorem{conjecture} {Conjecture}

\newcommand{\pre}{\mathrm{Re}}
\newcommand{\pim}{\mathrm{Im}}
\newcommand{\bi} {\ensuremath{{\bf i}}}
\newcommand{\bo} {\ensuremath{{\bf i_1}}}
\newcommand{\bos}{\ensuremath{{\bf i_1^{\text 2}}}}
\newcommand{\bts}{\ensuremath{{\bf i_2^{\text 2}}}}

\newcommand{\bj}{\ensuremath{{\bf j}}}

\newcommand {\bjp}{\ensuremath{{\bf j_1}}}
\newcommand {\bjps}{\ensuremath{{\bf j_1^{\text 2}}}}
\newcommand {\bjd}{\ensuremath{{\bf j_2}}}
\newcommand {\bjds}{\ensuremath{{\bf j_2^{\text 2}}}}

\newcommand{\bt} {\ensuremath{{\bf i_2}}}
\newcommand{\bb} {\ensuremath{{\bf i_3}}}
\newcommand{\bbs} {\ensuremath{{\bf i_3^{\text 2}}}}
\newcommand{\bq} {\ensuremath{{\bf i_4}}}
\newcommand{\bqs} {\ensuremath{{\bf i_4^{\text 2}}}}
\newcommand{\bn} {\ensuremath{{\bf i_n}}}

\newcommand{\bk} {\ensuremath{{\bf i_{k}}}}
\newcommand {\bjt}{\ensuremath{{\bf j_3}}}
\newcommand {\bjts}{\ensuremath{{\bf j_3^{\text 2}}}}
\newcommand {\bjk}{\ensuremath{{\bf j_k}}} %AJOUT
\newcommand{\bil} {\ensuremath{{\bf i_l}}}
\newcommand{\bim} {\ensuremath{{\bf i_m}}}
\newcommand{\eb} {\ensuremath{\gamma_1}}

\newcommand{\ett} {\ensuremath{\gamma_2}}
\newcommand{\etc} {\ensuremath{\Ol{\gamma}_2}}

\newcommand{\Ol}{\overline}

\newcommand{\mC}{\ensuremath{\mathbb{C}}}
\newcommand{\mD}{\ensuremath{\mathbb{D}}}
\newcommand{\mN}{\ensuremath{\mathbb{N}}}
\newcommand{\mZ}{\ensuremath{\mathbb{Z}}}
\newcommand{\mR}{\ensuremath{\mathbb{R}}}
\newcommand{\mT}{\ensuremath{\mathbb{T}}}

\newcommand{\mH}{\ensuremath{\mathbb{D}}}

\newcommand{\mM}{\ensuremath{\mathbb{M}}}
\newcommand{\mMan}{\ensuremath{\mathcal{M}}}
\newcommand{\mTet}{\ensuremath{\mathcal{T}}}

\newcommand{\op}{\left(}
\newcommand{\fp}{\right)}
\newcommand{\oa}{\left\{}
\newcommand{\fa}{\right\}}
\newcommand{\lc}{\left[}
\newcommand{\rc}{\right]}

\newcommand{\Qpit}{\ensuremath{\left\lbrace Q_{p,c}^m(0) \right\rbrace_{m=1}^{\infty}}}

\newcommand{\mManp}{\ensuremath{\mathcal{M}^p}}

\newcommand{\mHt}{\ensuremath{\mathcal{H}^2}}

\newcommand{\mHyb}{\ensuremath{\mathcal{H}}}
\newcommand{\cP}{\ensuremath{\mathcal{P}}}

\newcommand{\bHpc}{\ensuremath{\mathbf{H}_{p,c}}}

\newcommand{\VEC}[2]{\begin{pmatrix}
#1 \\ #2
\end{pmatrix}}

\newcommand{\biq}{\ensuremath{\bf i_q}}
\newcommand{\bis}{\ensuremath{\bf i_s}}

\begin{document}

%%%%%%%%%%%%%%%%
%%%%%%%%% TITLE, authors' names and adresses
\markboth{P.-O. Paris\'e and D. Rochon}{Tricomplex dynamical systems generated by polynomials of odd degree}

\title{Tricomplex dynamical systems generated by polynomials of odd degree}

\author{Pierre-Olivier Paris\'e \thanks{E-mail: \texttt{Pierre-Olivier.Parise@uqtr.ca}} \and
Dominic Rochon\thanks{E-mail: \texttt{Dominic.Rochon@UQTR.CA}}}
%\date{\today}
\date{Département de mathématiques et
d'informatique \\ Université du Québec à Trois-Rivières \\
C.P. 500 Trois-Rivières, Québec \\ Canada, G9A 5H7}

\maketitle

%%Abstract
\begin{abstract}
In this article, we give the exact interval of the cross section of the \textit{Multibrot} sets generated by the polynomial $z^p+c$ where $z$ and $c$ are complex numbers and $p > 2$ is an odd integer. Furthermore, we show that the same Multibrots defined on the hyperbolic numbers are always squares. Moreover, we give a generalized 3D version of the hyperbolic Multibrot set and prove that our generalization is an octahedron for a specific 3D slice of the dynamical system generated by the tricomplex polynomial $\eta^p+c$ where $p > 2$ is an odd integer.
\end{abstract}\vspace{0.5cm} 
\noindent\textbf{AMS subject classification:} 37F50, 32A30, 30G35, 00A69 \\
\textbf{Keywords:} Tricomplex dynamics, Multibrot, Hyperbrot, Generalized Mandelbrot sets, Multicomplex numbers, 3D fractals

%%%%%
%Introduction
\section*{Introduction}

Multicomplex dynamics appears for the first time in 2000 (see \cite{Rochon1} and \cite{Rochon2}). The author of these articles used a commutative generalization of complex numbers called the bicomplex numbers, denoted $\mM (2)$, to extend the well known Mandelbrot set in four dimensions and to give a 3D version of it.

The Multibrot sets were first studied in \cite{Gujar} and \cite{Papathomas}. Following these works, X.-Y. Wang and W.-J. Song \cite{Chine1} extend to the bicomplex space the Multibrot sets defined as
	\begin{equation}
	\mMan^p:= \oa c \in \mC \, | \, \oa Q_{p,c}^m(0) \fa_{m=1}^{\infty} \text{ is bounded} \fa \label{MultibrotDef}
	\end{equation}
where $Q_{p,c}(z) = z^p + c$. Recently, in  \cite{RochonParise}, the authors introduce the tricomplex Multibrot sets. Based on an other work on the topic (see \cite{GarantRochon}), they proved some topological properties of the complex, hyperbolic and tricomplex Multibrot sets and studied the Multibrot set generated by the polynomial $\eta^3+c$ where $\eta$ and $c$ are tricomplex numbers. Particularly, they proved that the real intersection of the Multibrot set $\mMan^3$ is the interval $\lc -\frac{2}{2\sqrt{3}}, \frac{2}{2\sqrt{3}} \rc$. Their proof was based on the exact formula of the roots of the polynomial $x^3-x+c$. However, we know since Galois' work that no general formula expresses, in terms  of the elementary operations, all the roots of a polynomial of degree greater or equal to 5. 

In this article, we use another approach to find the real intersection of the Multibrot sets as it is conjectured in \cite{RochonParise}:
	\begin{equation}
	\mMan^p \cap \mR = \lc -\frac{p-1}{p^{p/(p-1)}}, \frac{p-1}{p^{p/(p-1)}} \rc \label{princConjecture}
	\end{equation}
for any odd integer $p > 2$. We prove that this conjecture is true and has many consequences in tricomplex dynamics.

The article is separated into three sections. In the first section, we recall some basics of the theory of tricomplex numbers. In the second section, we introduce the tricomplex Multibrot sets where complex and hyperbolic dynamics are embedded. In the third section, we analyse the locus of the real roots of the polynomial $x^p-x+c$ for a given odd integer $p > 2$. Then, we apply those results to Multibrot sets to prove \eqref{princConjecture}. Moreover, we prove that the Hyperbrot set defined as
	\begin{equation}
	\mHyb^p:=\oa c \in \mD \, | \, \oa Q_{p,c}^m(0)\fa_{m=1}^{\infty } \text{ is bounded} \fa \label{HyperbrotDef}
	\end{equation}
where $\mD$ is the set of hyperbolic numbers is a square when the degree $p > 2$ is an odd integer. Finally, we define the generalized three-dimensional (3D) version of the Hyperbrot set of order $p$ as a 3D slice of a tricomplex Multibrot of order $p$ and prove that it is a regular octahedron for all odd integers $p>2$.

%%%%%
% Préliminaires sur nombres Tricomplexes
\section{Preliminaries}
\label{Pre}
In this section, we begin by introducing the tricomplex space $\mM (3)$. The reader may refer to \cite{Baley}, \cite{GarantPelletier}, \cite{Parise} and \cite{Vajiac} for more details on the next properties.

A tricomplex number $\eta$ is composed of two coupled bicomplex numbers $\zeta_1$, $\zeta_2$ and an imaginary unit $\bb$ such that
	\begin{equation}
	\eta=\zeta_1 + \zeta_2 \bb \label{eq2.1}
	\end{equation}
where $\bbs=-1$. The set of such tricomplex numbers is denoted by $\mM (3)$. Since $\zeta_1,\zeta_2 \in \mM (2)$, we can write them as $\zeta_1=z_1+ z_2\bt$ and $\zeta_2=z_3+ z_4\bt$ where $z_1,z_2,z_3,z_4 \in \mM (1)\simeq \mC$. In that way, \eqref{eq2.1} can be rewritten as
	\begin{equation}
	\eta=z_1+ z_2 \bt+  z_3 \bb+  z_4 \bjt\label{eq2.2}
	\end{equation}
where $\bts=-1$, $\bt \bb = \bb \bt = \bjt$ and $\bjts=1$. Moreover, as $z_1$, $z_2$, $z_3$ and $z_4$ are complex numbers (in $\bo$), we can write the number $\eta$ in a third form as
	\begin{align}
	\eta&=a+ b\bo + (c+  d\bo)\bt + (e +  f\bo)\bb + (g +  h\bo)\bjt\notag\\
	&=a+ b\bo +  c\bt +  d\bjp +  e\bb +  f\bjd + g \bjt +  h\bq\label{eq2.3}
	\end{align}
where $\bos=\bqs=-1$, $\bq =\bo \bjt = \bo \bt \bb$, $\bjd = \bo \bb = \bb \bo$, $\bjds=1$, $\bjp=\bo \bt = \bt \bo$ and $\bjps=1$. After ordering each term of \eqref{eq2.3}, we get the following representations of the set of tricomplex numbers:
	\begin{align}
	\mM (3) &:= \oa \eta = \zeta_1 +  \zeta_2\bb \, |\, \zeta_1, \zeta_2 \in \mM (2) \fa \notag\\
&=\oa z_1+ z_2 \bt+  z_3 \bb+ z_4 \bjt  \, |\, z_1,z_2,z_3,z_4 \in \mM (1) \fa \notag\\
&=\oa x_0+ x_1\bo + x_2\bt  +  x_3\bb + x_4\bq  +  x_5\bjp +  x_6 \bjd+  x_7\bjt \, \right. \notag \\ & \qquad \qquad \left. |\, x_i \in \mM (0)=\mR \text{ for } i=0,1,2, \ldots , 7 \fa \text{.} \label{EqRep}
	\end{align}
Let $\eta_1=\zeta_1+\zeta_2\bb$ and $\eta_2=\zeta_3 + \zeta_4\bb$ be two tricomplex numbers with $\zeta_1,\zeta_2,\zeta_3,\zeta_4 \in \mM (2)$. We define the equality, the addition and the multiplication of two tricomplex numbers as
	\begin{align}
	\eta_1&=\eta_2 \text{ iff } \zeta_1=\zeta_3 \text{ and } \zeta_2=\zeta_4 \label{eq2.4}\\
\eta_1 + \eta_2 &:= (\zeta_1 + \zeta_3) + (\zeta_2+\zeta_4)\bb \label{eq2.5}\\
\eta_1 \cdot \eta_2&:= (\zeta_1\zeta_3-\zeta_2\zeta_4)+(\zeta_1\zeta_4 + \zeta_2\zeta_3)\bb \label{eq2.6}\text{.}
	\end{align}
Table \ref{tabC1} shows the results after multiplying each tricomplex imaginary unity two by two.
	\begin{table}
	\centering
		\begin{tabular}{c|*{9}{c}}
		$\cdot$ & 1 & $\mathbf{i_1}$ & $\mathbf{i_2}$ & $\mathbf{i_3}$ & $\mathbf{i_4}$ & $\bjp$ & $\bjd$ & $\mathbf{j_3}$\\\hline
1 & 1 & $\mathbf{i_1}$ & $\mathbf{i_2}$ & $\mathbf{i_3}$ & $\mathbf{i_4}$ & $\bjp$ & $\bjd$ & $\mathbf{j_3}$\\
$\mathbf{i_1}$ & $\mathbf{i_1}$ & $-\mathbf{1}$ & $\bjp$ & $\bjd$ & $-\mathbf{j_3}$ & $-\mathbf{i_2}$ & $-\mathbf{i_3}$ & $\mathbf{i_4}$\\
$\mathbf{i_2}$ & $\mathbf{i_2}$ & $\bjp$ & $-\mathbf{1}$ & $\mathbf{j_3}$ & $-\bjd$ & $-\mathbf{i_1}$ & $\mathbf{i_4}$ & $-\mathbf{i_3}$\\
$\mathbf{i_3}$ & $\mathbf{i_3}$ & $\bjd$ & $\mathbf{j_3}$ & $-\mathbf{1}$ & $-\bjp$ & $\mathbf{i_4}$ & $-\mathbf{i_1}$ & $-\mathbf{i_2}$\\
$\mathbf{i_4}$ & $\mathbf{i_4}$ & $-\mathbf{j_3}$  & $-\bjd$ & $-\bjp$ & $-\mathbf{1}$ & $\mathbf{i_3}$ & $\mathbf{i_2}$ & $\mathbf{i_1}$\\
$\bjp$ & $\bjp$ & $-\mathbf{i_2}$  & $-\mathbf{i_1}$ & $\mathbf{i_4}$ & $\mathbf{i_3}$ & $\mathbf{1}$ & $-\mathbf{j_3}$ & $-\bjd$\\
$\bjd$ & $\bjd$ & $-\mathbf{i_3}$  & $\mathbf{i_4}$ & $-\mathbf{i_1}$ & $\mathbf{i_2}$ & $-\mathbf{j_3}$ &  $\mathbf{1}$ & $-\bjp$\\
$\mathbf{j_3}$ & $\mathbf{j_3}$ & $\mathbf{i_4}$ &$-\mathbf{i_3}$  & $-\mathbf{i_2}$ & $\mathbf{i_1}$ & $-\bjd$ & $-\bjp$ &  $\mathbf{1}$ \\
		\end{tabular}
	\caption{Products  of tricomplex imaginary units}\label{tabC1}
	\end{table}
The set of tricomplex numbers with addition $+$ and multiplication $\cdot$ forms a commutative ring with zero divisors.

A tricomplex number has a useful representation using the idempotent elements $\ett =\frac{1+\bjt}{2}$ and $\etc =\frac{1-\bjt}{2}$. Recalling that $\eta = \zeta_1 +  \zeta_2\bb$ with $\zeta_1, \zeta_2 \in \mM (2)$, the idempotent representation of $\eta$ is
	\begin{equation}
	\eta = (\zeta_1- \zeta_2\bt)\ett + (\zeta_1+ \zeta_2\bt)\etc \label{eq2.7}\text{.}
	\end{equation}
The representation \eqref{eq2.7} of a tricomplex number allows to add and multiply tricomplex numbers term-by-term. In fact, we have the following theorem (see \cite{Baley}):
	\begin{theorem}\label{theo2.2}
	Let $\eta_1=\zeta_1 +  \zeta_2\bb$ and $\eta_2=\zeta_3 +  \zeta_4\bb$ be two tricomplex numbers. Let $\eta_1=u_1\ett + u_2 \etc$ and $\eta_2=u_3\ett + u_4\etc$ be the idempotent representation \eqref{eq2.7} of $\eta_1$ and $\eta_2$. Then,
	\begin{enumerate}
\item $\eta_1+\eta_2=(u_1+u_3)\ett + (u_2+u_4)\etc$;
\item $\eta_1 \cdot \eta_2 = (u_1 \cdot u_3)\ett + (u_2 \cdot u_4)\etc$;
\item $\eta_1^m=u_1^m \ett + u_2^m \etc$ $\forall m \in \mN$.
	\end{enumerate}
	\end{theorem}
Moreover, we define a $\mM (3)$-\textit{Cartesian} set $X$ of two subsets $X_1,X_2\subseteq\mM (2)$ as follows:
	\begin{align}
	X=X_1\times_{\ett}X_2&:=\oa \eta =\zeta_1+ \zeta_2\bb \in \mM (3) \, | \, \eta = u_1 \ett + u_2 \etc , \right. \notag\\ 
	& \qquad \left. u_1 \in X_1 \text{ and } u_2 \in X_2 \fa \text{.}\label{eq2.14}
	\end{align}

Let define the norm $\Vert \cdot \Vert_3 :\, \mM (3) \rightarrow \mR$ of a tricomplex number $\eta=\zeta_1 + \zeta_2\bb$ as
	\begin{align}
	\Vert \eta \Vert_3  := \sqrt{\Vert\zeta_1\Vert_2^2+\Vert\zeta_2\Vert_2^2}=\sqrt{\sum_{i=1}^2|z_i|^2+\sum_{i=3}^4|z_i|^2}
=\sqrt{\sum_{i=0}^7x_i^2}.\label{eq2.15}
	\end{align}
We also have that $\Vert \eta \Vert_3 = \frac{\sqrt{\Vert \zeta_1 - \zeta_2 \bt \Vert_2^2 + \Vert \zeta_1 + \zeta_2 \bt \Vert_2^2}}{2}$. According to the Euclidean norm \eqref{eq2.15}, we say that a sequence $\oa s_m \fa_{m=1}^{\infty} $ of tricomplex numbers is bounded if and only if there exists a real number $M$ such that $\Vert s_m \Vert_3 \leq M$ for all $m \in \mN$. Now,
according to \eqref{eq2.14}, we define two kinds of tricomplex discus:
	\begin{definition}\label{def2.1}
	Let $\alpha = \alpha_1+\alpha_2 \bb \in \mM (3)$ and set $r_2\geq r_1 > 0$.
	\begin{enumerate}
	\item The open discus is the set
	\begin{align}
	{\bf{D_3}}(\alpha ; r_1,r_2)&:= \oa  \eta \in \mM (3) \, | \, \eta =\zeta_1 \ett + \zeta_2\etc , \, \Vert \zeta_1-(\alpha_1- \alpha_2 \bt)\Vert_2 <r_1 \text{ and }\right.\notag \\ &\left. \qquad \qquad \qquad \Vert \zeta_2 - (\alpha_1+\alpha_2 \bt) \Vert_2 <r_2 \fa \text{.}\label{eq2.141}
	\end{align}
	\item The closed discus is the set
	\begin{align}
	\Ol{\bf{D_3}}(\alpha ; r_1,r_2)&:= \oa  \eta \in \mM (3) \, | \, \eta =\zeta_1 \ett + \zeta_2\etc , \, \Vert \zeta_1-(\alpha_1- \alpha_2\bt)\Vert_2 \leq r_1 \text{ and }\right.\notag \\ &\left. \qquad \qquad \qquad \Vert \zeta_2 - (\alpha_1+ \alpha_2 \bt) \Vert_2 \leq r_2 \fa \text{.}\label{eq2.142}
	\end{align}
	\end{enumerate}
	\end{definition}

We end this section by several remarks about subsets of $\mM (3)$. Let the set $\mC (\bk ):=\oa \eta = x_0 + x_1 \bk \, | \, x_0, x_1 \in \mR \fa, \bk \in \oa \bo , \bt , \bb , \bq \fa $. So, $\mC (\bk )$ is a subset of $\mM (3)$ for $k=1,2,3,4$ and we also remark that they are all isomorphic to $\mC$. Furthermore, the set $\mH (\bjk ):=\oa x_0 + x_1 \bjk \, | \, x_0, x_1 \in \mR\fa$ where  $\bjk \in \oa \bjp , \bjd, \bjt  \fa$ is a subset of $\mM (3)$ and is isomorphic to the set of hyperbolic numbers $\mH$ for $k\in\{1,2,3\}$ (see \cite{RochonShapiro, vajiac2} and \cite{Sobczyk} for further details about the set $\mH$ of hyperbolic numbers). The $\bjk$-part of a hyperbolic number is $\mathrm{Hy} ( x_0 + x_1 \bjk ) = x_1$. Finally, we recall (see \cite{RochonParise} and \cite{GarantRochon}) another important subset of $\mM (3)$ useful in the next section.
	\begin{definition}\label{Tikilim}
	Let $\bk , \bil , \bim \in \oa 1, \bo , \bt , \bb , \bq , \bjp , \bjd , \bjt \fa$ with $\bk \neq \bil$, $\bk \neq \bim$ and $\bil \neq \bim$. We define a 3D subset of $\mM (3)$ as
		\begin{equation}
		\mT (\bim , \bk , \bil ):= \oa x_1\bk + x_2\bil + x_3\bim \, | \, x_1,x_2,x_3 \in \mR \fa \text{.}
		\end{equation}
	\end{definition}

\section{Tricomplex dynamics}
In section \ref{Pre}, we mentioned that the set of complex numbers is isomorphic to the set $\mC (\bk )$ where $\bk \in \oa \bo , \bt , \bb , \bq \fa$ and the set of hyperbolic numbers is isomorphic to $\mD (\bjk )$ where $\bjk \in \oa \bjp , \bjd , \bjt \fa$. So, the Multibrot sets, the Hyperbrot sets and the bicomplex Multibrot sets as defined in \cite{Chine1} and \cite{RochonParise} are generalized by defining the tricomplex Multibrot sets.
	\begin{definition}\label{d3.1}
	Let $Q_{p,c}(\eta )=\eta^p+c$ where $\eta , c \in \mM (3)$ and $p\geq 2$ an integer. The tricomplex \textit{Multibrot} set is defined as the set
		\begin{equation}
		\mMan_3^p:=\oa c \in \mM (3) \, | \, \Qpit \text{ is bounded } \fa \text{.}
		\end{equation}
	\end{definition}

	\begin{remark}\label{rem3.1}
If we restrict the components of the tricomplex number $c$ to the subspace $\mC (\bo ) \simeq \mC$, we see that the set $\mMan_3^p$ is equal (in $\mR^2$) to the Multibrot set $\mMan^p$ as defined in the introduction. Moreover, if we restrict the components of the number $c$ to the subspace $\mH (\bjp )$, we see that the set $\mMan_3^p$ is equal (in $\mR^2$) to the hyperbolic Multibrot set $\mHyb^p$ also defined in the introduction. So, for the rest of this article, we identify the $\mMan^p$ and $\mHyb^p$ sets with the two following sets:
		\begin{align*}
		 \mMan^p = \oa c = x_0 + x_1\bo \, | \, \Qpit \text{ is bounded } \fa \\
		 \text{ and }\mHyb^p = \oa c = x_0 + x_5\bjp \, | \, \Qpit \text{ is bounded } \fa.
		\end{align*}
	\end{remark}
	
The Multibrot set $\mMan^p$ has the following important property (see \cite{RochonParise}).
	\begin{theorem}\label{t2.2.2}
	A complex number $c$ is in $\mManp$ if and only if $|Q_{p,c}^m(0)|\leq 2^{1/(p-1)}$ $\forall m \in \mN$.
	\end{theorem}
	
Given an integer $p \geq 2$, the set $\mMan_3^p$ is a closed and connected set, and also it is contained in the closed discus $\overline{\mathbf{D}}(0,2^{1/(p-1)}, 2^{1/(p-1)})$. In the last part of the article, those information are used to generate the images of several 3D slices. From \cite{RochonParise}, $\mMan_3^p$ can be viewed as the tricomplex Cartesian product
	\begin{equation}
	\mMan_3^p = \mMan_2^p \times_{\ett} \mMan_2^p
	\end{equation}
where $\mMan_2^p$ is the bicomplex Multibrot set. Moreover, the set $\mMan_3^p$ can be rewritten as the following mixed Cartesian product
	\begin{equation}
	\mMan_3^p = (\mMan_1^p \times_{\eb} \mMan_1^p) \times_{\ett} (\mMan_1^p \times_{\eb} \mMan_1^p) 
	\end{equation}
where the set $\mMan_1^p$ is the Multibrot set in the complex plane.

To visualize the tricomplex Multibrot sets, we define a \textbf{principal 3D slice} as follows
	\begin{equation}\label{eq5.2.1}
	\mTet^p:= \mTet^p(\bim , \bk , \bil ) =\oa c \in \mT (\bim , \bk , \bil ) \, | \, \oa Q_{p,c}^m(0) \fa_{m=1}^{\infty} \text{ is bounded } \fa
\text{.}
	\end{equation}
So the number $c$ has three of its components that are not equal to zero. In total, there are 56 possible combinations of principal 3D slices. To attempt a classification of these slices, we introduce a relation $\sim$ (see \cite{GarantRochon}).

\begin{definition}\label{def5.2.1}
Let $\mTet_1^p( \bim , \bk , \bil )$ and $\mTet_2^p( \bn , \biq , \bis )$ be two principal 
3D slices of a tricomplex Multibrot set $\mathcal{M}_3^p$. Then, $\mTet_1^p \sim \mTet_2^p$ 
if we have a bijective linear mapping $\varphi:\mM (3)$ $\rightarrow$ $\mM (3)$ such that 
$\forall c_2 \in \mT(\bn , \biq , \bis )$ there exists a $c_1 \in \mT(\bim , \bk , \bil )$ with $\varphi(c_1)=c_2$
and $$(\varphi \circ Q_{p,c_1} \circ \varphi^{-1})(\eta )=Q_{p,c_2}(\eta )\mbox{ }\forall \eta\in \mM (3).$$ 
In that case, we say that $\mTet_1^p$ and $\mTet_2^p$ have the same dynamics.
\end{definition}

If two 3D slices are in relationship in term of $\sim$, then we also say that they are symmetrical. This comes from the fact that their visualizations by a computer give the same images, but oriented differently. In \cite{Parise}, it is showed that $\sim$ is also an equivalent relation on the set of principal 3D slices of $\mMan_3^p$.

\section{Roots of polynomial and applications}
In this section, we show some applications of the locus of the roots of polynomials to complex and hyperbolic dynamics. In addition, we give a generalized 3D version of the Hyperbrot sets and prove that our generalization is an octahedron in the case of the polynomial $Q_{p,c}$ with $ p > 2$ an odd integer. We first state some results on the locus of the roots of the polynomial $x^p - x + c$. These results will be useful to prove the Theorem \ref{t2.3.1} which states that the intersection of $\mMan^p$ and the real axis is exactly $\left[ -\frac{p-1}{p^{p/(p-1)}}, \frac{p-1}{p^{p/(p-1)}} \right]$. 

	\subsection{Behavior of the roots}
Let $R_{p,c}(x) = x^p-x+c$ where $x,c\in \mR$. In this part, we study the exact location of the real roots of $R_{p,c}$. We also let $p > 2$ be an odd integer.

%%%%%%%%%%%%%%%%%
%%% Study of the function R_{p,c}

We start by computing $R'_{p,c}(x) = px^{p-1} - 1$ where $p-1$ is even. Letting $R'_{p,c}(x) = 0$ gives the following real critical points of $R_{p,c}(x)$
	\begin{equation}
	w_1 = -\frac{1}{p^{1/(p-1)}} \text{ and } w_2 = \frac{1}{p^{1/(p-1)}}\text{.}\label{critPoints}
	\end{equation}
Moreover, $R'_{p,c}(x) > 0$ when $|x| > \frac{1}{p^{1/(p-1)}}$ and $R'_{p,c}(x) < 0$ when $|x| < \frac{1}{p^{1/(p-1)}}$.  The values taken by $R_{p,c}(x)$ at the critical points are
	\begin{align}
	R_{p,c}(w_1)& = \left( \dfrac{p-1}{p^{p/(p-1)}} \right) + c,\label{w1V}\\
	R_{p,c}(w_2)& = -\left( \dfrac{p-1}{p^{p/(p-1)}} \right) + c \label{w2V}\text{.}
	\end{align}
Furthermore, $R''_{p,c}(x) = p(p-1)x^{p-2}$ with $p-2 \geq 1$ an odd integer. Then, we have that
	\begin{align*}
	R''_{p,c}(x) < 0 \text{ and } R''_{p,c}(x) > 0 
	\end{align*}
for $x < 0$ and $x > 0$ respectively, and so
	\begin{align*}
	R''_{p,c}\left( w_1\right) < 0 \text{ and } R''_{p,c}\left( w_2 \right) > 0\text{.}
	\end{align*}
	
Thus, $R_{p,c}(x)$ is increasing on $\left( -\infty , w_1 \right) \cup \left( w_2, +\infty \right)$ and decreasing on $\left( w_1 , w_2 \right)$. $R_{p,c}(x)$ has a local maximum at $w_1$ given by \eqref{w1V} and has a local minimum at $w_2$ given by \eqref{w2V}. Following that behavior study of $R_{p,c}(x)$, we give a first lemma that characterized the roots for $|c| < \frac{p-1}{p^{p/(p-1)}}$.

%%%%% results on the location of roots
%%%%%%%%%%%%%%%%%%%%%%%%%%%%%%%%%%%%%%%
	\begin{lemma}\label{l2.1.1}
	Let $p > 2$ be an odd integer. If $c \in \left(-\frac{p-1}{p^{p/(p-1)}},\; \frac{p-1}{p^{p/(p-1)}}\right)$ then $R_{p,c}$ has three distinct real roots $a_1$, $a_2$ and $a_3$ such that
		\begin{align*}
		a_1 \in \left( -\infty ,\; w_1 \right)  \text{, } 
		a_2 \in \left( w_1,\; w_2\right)\text{ and } a_3 \in \left( w_2,\; +\infty \right) \text{.}&
		\end{align*}
	\end{lemma}
	\begin{proof}
	Since $R_{p,c}$ is continuous and $\lim\limits_{x \rightarrow -\infty} R_{p,c}(x) = -\infty$, $R_{p,c}(w_1)>0$, 
	$R_{p,c}(w_2)$ $<0$ and 
	$\lim\limits_{x \rightarrow +\infty}  R_{p,c}(x) = \infty$ for $c \in \left(-\frac{p-1}{p^{p/(p-1)}},\; \frac{p-1}{p^{p/(p-1)}}\right)$, there exist three distinct zeros $a_1$, $a_2$ and $a_3$ situated as desired.
	\end{proof}

	\begin{remark}\label{rem3.2.1} 
	Suppose that $|c| < \frac{p-1}{p^{p/(p-1)}}$. In that case, since $R_{p,c}$ is monotonic, we obtain that
	\begin{align}
	R_{p,c}(x) > 0 & \quad \forall x \in \left( a_1, w_1\right)  \label{eqpos}\\
	\text{and } R_{p,c}(x) < 0 & \quad \forall x \in \left( w_2,a_3\right) \label{eqneg}\text{.}
	\end{align}
	Therefore, $a_1$, $a_2$ and $a_3$ are uniquely determined.
	\end{remark}

We can get a better approximation of the roots $a_1$ and $a_3$ depending on the sign of the number $c$. The next lemma encapsulates this new approximations.

	\begin{lemma}\label{l2.1.2}
	Let $R_{p,c}(x) = x^p-x+c$ where $x,c \in \mR$ and $p$ is an odd integer greater than 2. Then
	\begin{itemize}
	\item[1.] if $c = 0$, then $a_1 = -1$ and $a_3 = 1$;
	\item[2.] if $c \in \left( -\frac{p-1}{p^{p/(p-1)}}, 0 \right)$, then $a_1 \in \left( -1,w_1\right)$;
	\item[3.] if $c \in \left( 0, \frac{p-1}{p^{p/(p-1)}} \right)$, then $a_3 \in \left(w_2,1\right)$.
	\end{itemize}
	\end{lemma}
	\begin{proof}
	From Lemma \ref{l2.1.1}, we know that $R_{p,c}(x)$ has three distinct real roots $a_1$, $a_2$ and $a_3$. Let consider the special case where $c = 0$. In this case, we have that $R_{p,c}(x) = x^p-x = x(x^{p-1}-1)$ which implies that $R_{p,c}(x) = 0$ if and only if $a_1 = -1$, $a_2 = 0$ and $a_3 = 1$. For the second case, suppose by contradiction that $a_1 \leq -1$ for a certain $c \in  \left( -\frac{p-1}{p^{p/(p-1)}}, 0 \right)$. If $a_1 = -1$, then we have that
		\begin{align*}
 		a_1^p-a_1+c=R_{p,c}(a_1) = 0 = R_{p,0}(a_1)=a_1^p-a_1
		\end{align*}
which implies that $c = 0$. This is a contradiction. Suppose now that $a_1 < -1$. Since $c < 0$, we see that $R_{p,c}(x) = x^p-x+c < x^p-x = R_{p,0}(x)$ $\forall x \in \mR$. In particular, this inequality holds on the interval $( a_1 , w_1)$. So, from \eqref{eqpos}, we obtain that $R_{p,0}(x) > 0$ on the interval $( a_1 , w_1)$. In particular, $R_{p,0}(-1) > 0$ since $a_1 < -1 < w_1$. However, $R_{p,0}(-1) = 0$ from the first part, this is a contradiction. Thus, we have that $a_1 > -1$.

	The third case is handled similarly as the second one by supposing that $a_3 \geq 1$ which leads to a contradiction.
	\end{proof}

The next lemma explains what is happening at the boundary of the interval $\left[ -\frac{p-1}{p^{p/(p-1)}}, \frac{p-1}{p^{p/(p-1)}} \right]$.

	\begin{lemma}\label{l2.1.3}
	Let $R_{p,c}(x) = x^p-x+c$ where $x,c \in \mR$ and $p$ is an odd integer greater than 2. Then
	\begin{itemize}
	\item[1.] if $c = -\frac{p-1}{p^{p/(p-1)}}$, then $R_{p,c}$ has two real roots $w_1$ and $a_3$ where $w_1$ has multiplicity 2 and $a_3 \in \left( w_2, +\infty \right)$;
	\item[2.] if $c = \frac{p-1}{p^{p/(p-1)}}$, then $R_{p,c}$ has two real roots $w_2$ and $a_1$ where $w_2$ has multiplicity 2 and $a_1 \in \left( -\infty , w_1\right)$.
	\end{itemize}
	\end{lemma}
	\begin{proof}
	We know that 
		\begin{align*}
		R_{p,c}(w_1) = \dfrac{p-1}{p^{p/(p-1)}} + c \text{ and } R_{p,c}(w_2) = -\dfrac{p-1}{p^{p/(p-1)}} + c\text{.}
		\end{align*}

	For the first case, if $c = -\frac{p-1}{p^{p/(p-1)}}$, then $R_{p,c}(w_1) = 0$ and $R'_{p,c}(w_1) = 0$. Thus, $w_1$ is a root of multiplicity 2. We conclude that $a_1 = a_2 = w_1$ and $a_3 > w_2$. Thus, $R_{p,c}(x)$ has two real roots.

	For the second case, if $c = \frac{p-1}{p^{p/(p-1)}}$, then $R_{p,c}(w_2) = 0$ and $R'_{p,c}(w_2) = 0$. Thus, $w_2$ is a root of multiplicity 2. We conclude that $a_1 < w_1$ and $a_2 = a_3 = w_2$. Thus, $R_{p,c}(x)$ has two real roots.
	\end{proof}

Finally, the next lemma tells us how behave the roots of $R_{p,c}$ when $|c| > \frac{p-1}{p^{p/(p-1)}}$.

	\begin{lemma}\label{l2.1.4}
	Let $R_{p,c}(x) = x^p-x+c$ where $x,c \in \mR$ and $p$ is an odd integer greater than 2. Then
	\begin{itemize}
	\item[1.] if $c < -\frac{p-1}{p^{p/(p-1)}}$, then $R_{p,c}(x)$ has exactly one positive root $a_3 > w_2$;
	\item[2.] if $c > \frac{p-1}{p^{p/(p-1)}}$, then $R_{p,c}(x)$ has exactly one negative root $a_1 < w_1$.
	\end{itemize}
	\end{lemma}
	\begin{proof}
	For the first case, suppose that $c < -\frac{p-1}{p^{p/(p-1)}}$. We easily verify that $R_{p,c}(w_1) < 0$ and $R_{p,c}(w_2)  < 0$. Since $w_1$ is a local maximum and $R_{p,c}(x)$ is increasing on $(-\infty , w_1)$ and is decreasing on the interval $(w_1,w_2)$, we have that $R_{p,c}(x) < 0$ for all $x \in (-\infty , w_2)$. So, there exists only one positive real number $a_3 \in \left( w_2, +\infty \right)$ such that $R_{p,c}(a_3) = 0$. 

	For the second case, suppose that $c > \frac{p-1}{p^{p/(p-1)}}$. We easily verify that $R_{p,c}(w_1) > 0$ and  $R_{p,c}(w_2)  > 0$. Since $w_2$ is a local minimum and $R_{p,c}(x)$ is decreasing on the interval $(w_1,w_2)$ and increasing on $(w_2, +\infty )$, we have that $R_{p,c}(x) > 0$ for all $x \in (w_1 , +\infty )$. Hence, there exists only one negative real number $a_1 \in \left( -\infty, w_1\right)$ such that $R_{p,c}(a_1) = 0$. 
	\end{proof}

	\subsection{Real intersection of the Multibrot sets}
	We first establish two lemmas. The first one is regarding the symmetries of Multibrot sets (see \cite{Lau} and \cite{Sheng}).
	\begin{lemma}\label{l2.3.1}
	Let $c \in \mMan^p$ with $p\geq 2$ an integer and $c = |c|e^{\bi\theta_c}$. Then, $c_k := |c|e^{\bi \left( \theta_c + \frac{2k\pi}{p-1} \right)}$ is in $\mMan^p$ for any $k\in \mZ$.
	\end{lemma}
	\begin{proof}
	Suppose $c \in \mMan^p$. We infer by induction that $Q_{p,c_k}^m(0) = e^{\bi\frac{2k\pi}{p-1}} Q_{p,c}^m(0)$ $\forall m \geq 1$. Indeed, let $k \in \mZ$. First, we have that $Q_{p,c_k}(0) = c_k = |c| e^{\bi \left(\theta_c + \frac{2k\pi}{p-1}\right)} = e^{\bi  \frac{2k\pi}{p-1} } c$. Now, suppose that the assumption is true for an integer $m\geq 1$. Then,
		\begin{align*}
		Q_{p,c_k}^{m+1}(0) &= \left( Q_{p,c_k}^m(0) \right) ^p + c_k \\
		& = e^{\bi\left( \frac{2kp\pi}{p-1} \right)} \left(Q_{p,c}^m(0) \right)^p + e^{\bi \frac{2k\pi}{p-1}} c\\
		&= e^{\bi2k\pi}e^{\bi \frac{2k\pi}{p-1} } \left(Q_{p,c}^m(0) \right)^p + e^{\bi \frac{2k\pi}{p-1}} c\\
		&= e^{\bi \frac{2k\pi}{p-1}}  Q_{p,c}^{m+1}(0)\text{.}
		\end{align*}
Moreover, by Theorem \ref{t2.2.2}, $|Q_{p,c}^m(0)| \leq 2^{1/(p-1)}$ $\forall m \in \mN$. Hence, since $|Q_{p,c_k}^m(0)|=|Q_{p,c}^m(0)|$ $\forall m \in \mN$, $c_k \in \mMan^p$.
	\end{proof}
	
	\begin{remark}\label{rem3.4.1}
	We remark that for any $k\in \mZ$, we have $k = n(p-1) + t$ where $n\in \mZ$ and $t \in \oa 0, 1, \ldots , p-2 \fa$. Hence, we see that $e^{\bi \left( \theta_c + \frac{2k\pi}{p-1} \right)} = e^{\bi\left( \theta_c + \frac{2t\pi}{p-1} \right) }$ and there is a cycle of length $p-1$. This shows that there is a group of $p-1$ rotations about the origin attached to the set $\mMan^p$.
	\end{remark}
	The second Lemma is about some properties of the sequence $\Qpit$ when $c > 0$.
	\begin{lemma}\label{lemIncSeqPoly}
	Set $c>0$ where $c$ is a real number. Then, the sequence $\Qpit$ is strictly increasing. Furthermore, if the sequence $\Qpit$ is bounded, then it converges to $c_0>0$.
	\end{lemma}
	\begin{proof}
	Let $c \in \mR$ such that $c > 0$. Since $c > 0$, we have that $Q_{p,c}^m(0) > 0$ $\forall m \geq 1$. Moreover, it is easy to see that $Q_{p,c}$ is an increasing function on $[0, +\infty )$.
	
	Now, we prove by induction that $\Qpit$ is strictly increasing. For the case $m = 1$, we have
		\begin{align*}
		Q_{p,c} (0) = c < c^p + c = Q_{p,c}^2 (0)
		\end{align*}
	since $c > 0$. Suppose that the assumption is true for an integer $k > 0$, that is $Q_{p,c}^k (0) < Q_{p, c}^{k + 1} (0)$. Then,
		\begin{align*}
		Q_{p, c}^{k+2} (0) = (Q_{p,c}^{k + 1}(0))^p + c > (Q_{p,c}^{k}(0))^p + c
		\end{align*}
	by the induction hypothesis and the fact that $Q_{p,c}$ is an increasing function on $[0, +\infty )$. So, $Q_{p,c}^{k + 2} (0) > Q_{p,c}^{k+1} (0)$. Thus, $Q_{p,c}^{m+1} (0) > Q_{p,c}^m (0)$ $\forall m \geq 1$.
	
	If the sequence $\Qpit$ is bounded, then since it is an increasing sequence, it converges to a limit $c_0 > 0$.
	\end{proof}

We can state the following theorem concerning the set $\mR \cap \mManp$ for odd integers greater than two. Fig. \ref{figureMultibrots} is an illustration of some Multibrots. 
	\begin{theorem}\label{t2.3.1}
	Let $p$ be an odd integer greater than 2. The set $\mMan^p$ crosses the real axis on the interval
$\left[- \frac{(p-1)}{p^{p/(p-1)}},\, \frac{(p-1)}{p^{p/(p-1)}}\right]$.
	\end{theorem}
	\begin{proof}
	By the Lemma \ref{l2.3.1} and the Remark \ref{rem3.4.1}, if we let $t = \frac{p-1}{2}$, we see that if $c \in \mMan^p$, then $-c \in \mMan^p$. So,  we can restrict our proof to the interval $\left[ 0,\, \frac{p-1}{p^{p/(p-1)}} \right]$. Let $R_{p,c}(x)=x^p-x+c$ where $x,c \in \mR$. We start by showing that no point $c>\frac{(p-1)}{p^{p/(p-1)}}$ lies in $\mMan^p$. In this case, from Lemma \ref{l2.1.4}, $R_{p,c}$ has exactly one real root $a_1 < 0$ . Suppose that $c \in \mMan^p$, \textit{i.e.} $\oa Q_{p,c}^m(0)\fa_{m=1}^{\infty}$ is bounded. Then, Lemma 3 in \cite{RochonParise} implies that $\oa Q_{p,c}^m(0)\fa_{m=1}^{\infty}$ is strictly increasing and it converges to $c_0>0$. Since $Q_{p,c}(0)$ is a polynomial function, we have that
		\begin{align}
		c_0=\lim_{m \rightarrow \infty} Q_{p,c}^{m+1}(0)= Q_{p,c}\left( \lim_{m \rightarrow \infty} Q_{p,c}^{m}(0) \right)= Q_{p,c}(c_0)\text{.}
		\end{align}
Thus, $c_0$ is a real root of $R_{p,c}$ and from Lemma \ref{l2.1.4} we have $c_0=a_1$. However, since $c_0 > 0$, this is a contradiction with $a_1 < 0$. Thus, $c \not \in \mMan^p$.

	Next, we show that for $0 \leq c \leq \frac{p-1}{p^{p/(p-1)}}$, we have $c\in \mMan^p$. Obviously, $c=0$ is in $\mMan^p$. Suppose that $0 < c \leq \frac{p-1}{p^{p/(p-1)}}$. In this case, from Lemmas \ref{l2.1.2} and \ref{l2.1.3}, $R_{p,c}(x)$ has a real positive root $a_3 \in \left[ w_2 , 1\right)$. We prove by induction that $|Q_{p,c}^m(0)|<a_3$ $\forall m \in \mN$. For $m=1$, we have that $|Q_{p,c}(0)|=|c|<a_3$ because $|c| < w_2 \leq a_3$. Now, suppose that $|Q_{p,c}^k(0)| < a_3$ for a $k \in \mN$. Then, since $R_{p,c}(a_3)=a_3^p-a_3+c=0$ and $c>0$,
		\begin{equation*}
		|Q_{p,c}^{k+1}(0)|=|(Q_{p,c}^k(0))^p+c|\leq |(Q_{p,c}^k(0))|^p+|c| < a_3^p+c=a_3.
		\end{equation*}
Thus, the proposition is true for $k+1$ and $|Q_{p,c}^m(0)|<a_3$ $\forall m \in \mN$. Since $a_3 < 1 < 2^{1/(p-1)}$, then by the Theorem \ref{t2.2.2} we have $c \in \mMan^p$ .

	In conclusion, $\mMan^p \cap \mR_+ = \left[ 0, \, \frac{(p-1)}{p^{p/(p-1)}} \right]$.
	\end{proof}

\begin{figure}
\centering
\subfigure[$\mMan^3$]{
	\includegraphics[scale = 0.1]{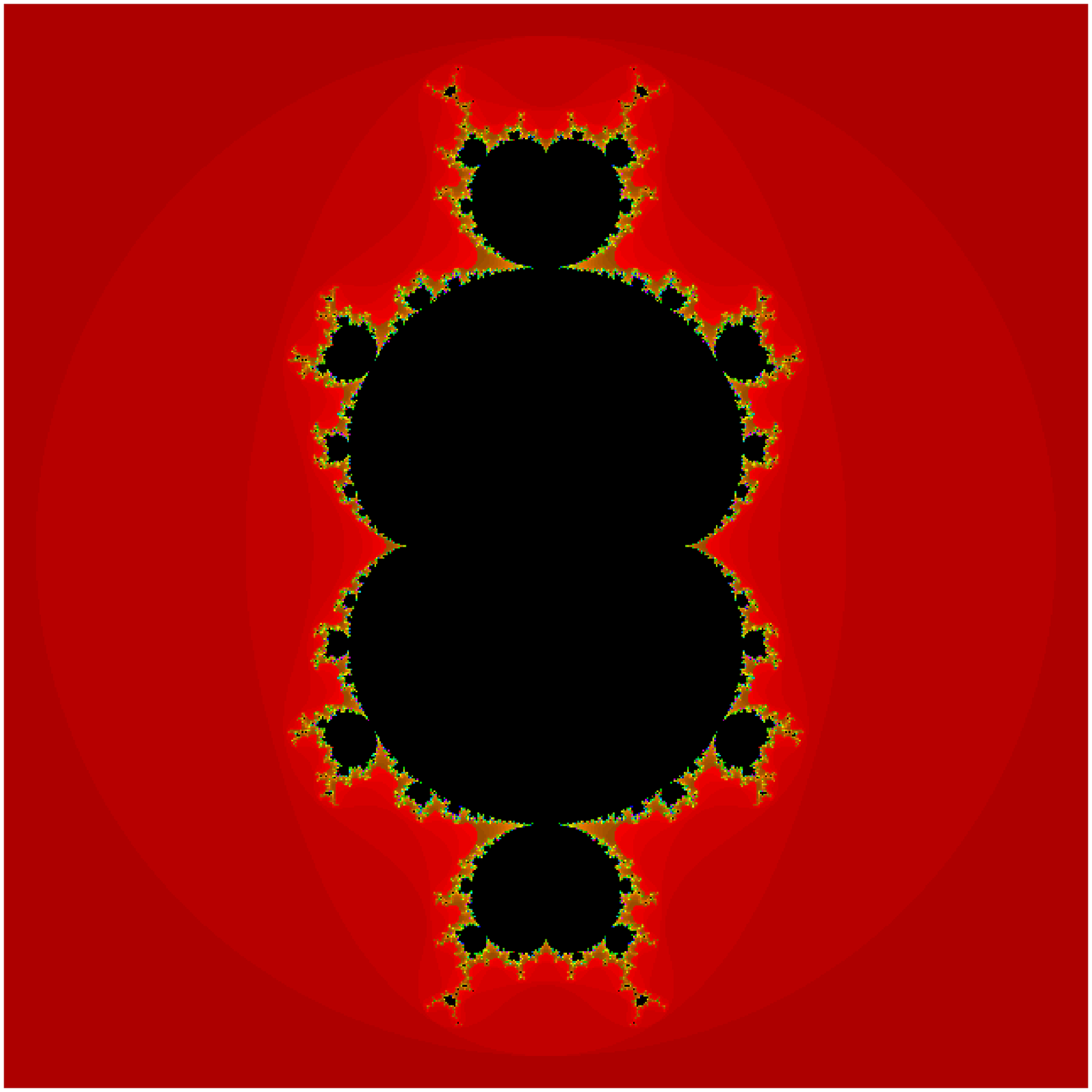}}
\subfigure[$\mMan^5$]{
	\includegraphics[scale = 0.1]{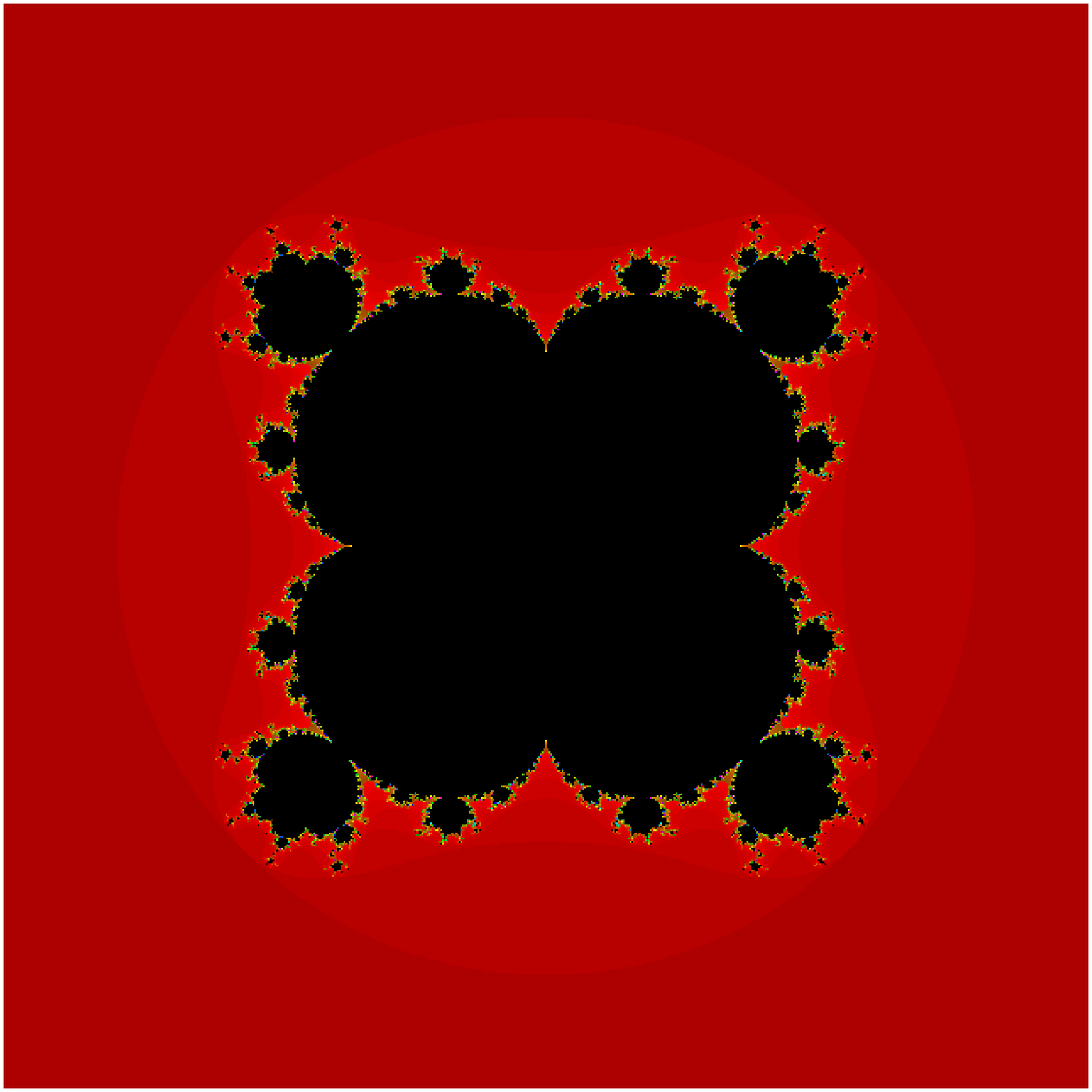}}
\subfigure[$\mMan^7$]{
	\includegraphics[scale = 0.1]{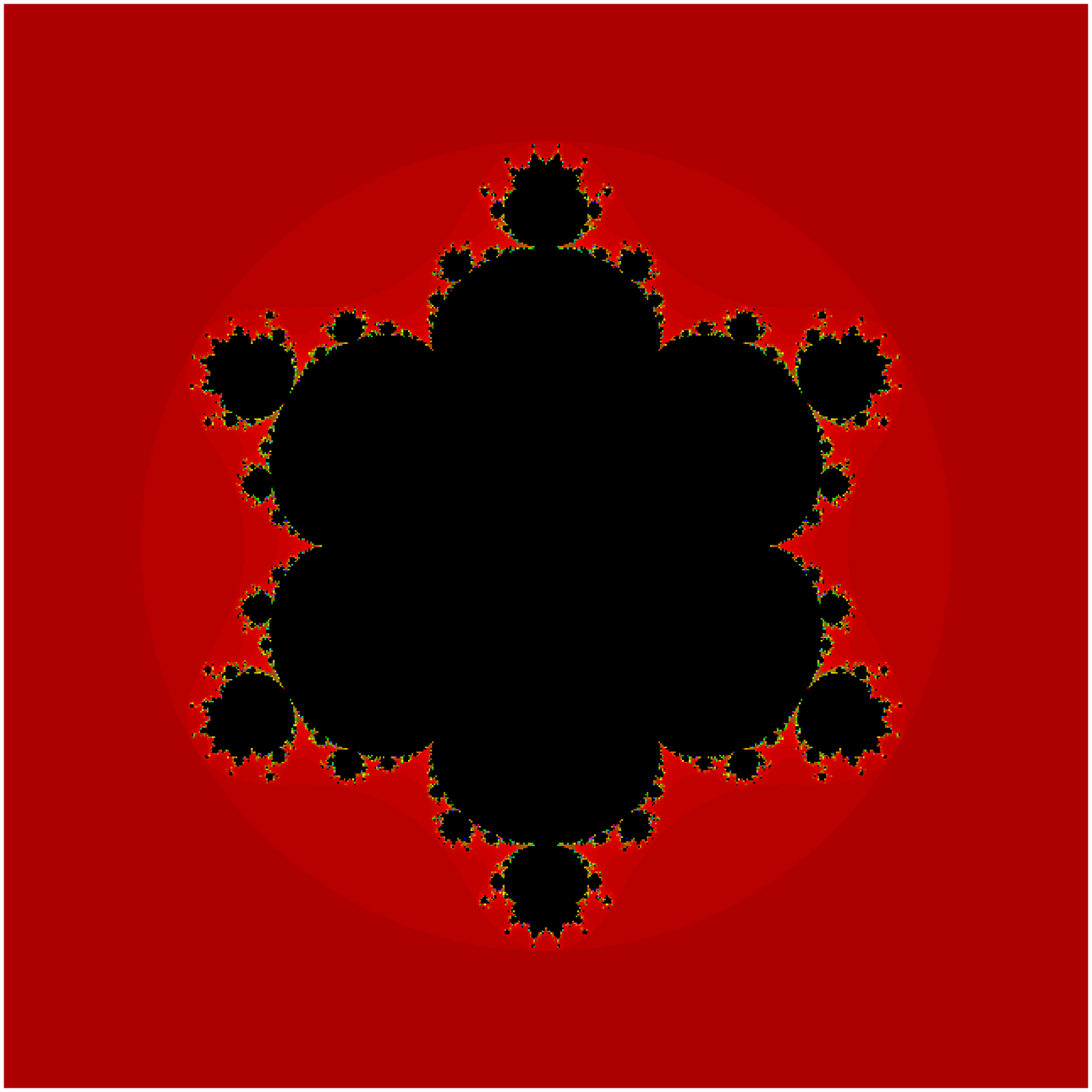}}
\subfigure[$\mMan^9$]{
	\includegraphics[scale = 0.1]{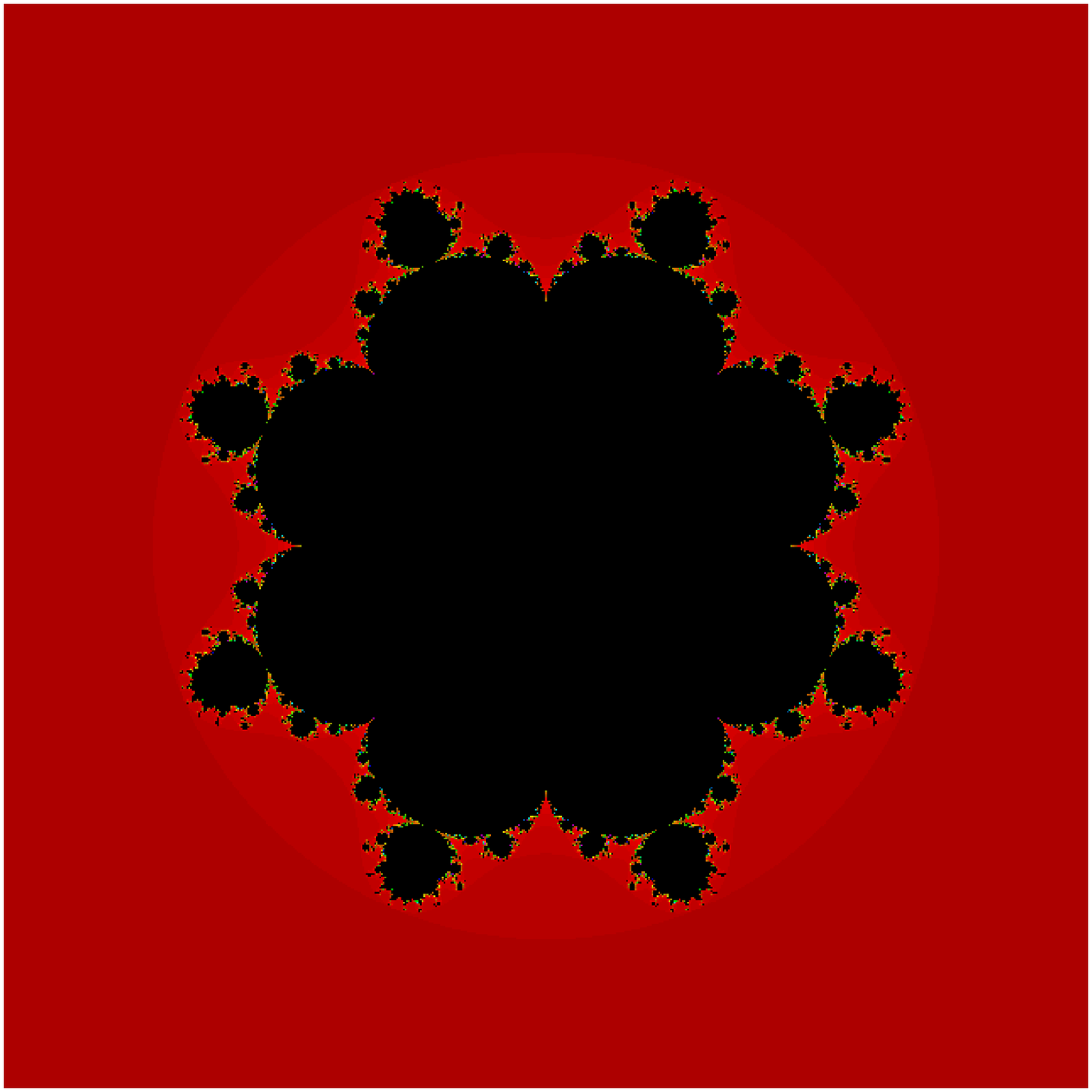}}
\caption{Multibrots for several odd integers, $-1.5 \leq \pre (c), \pim (c) \leq 1.5$}\label{figureMultibrots}
\end{figure}
	
	\subsection{Characterization of the Hyperbrots}\label{sec3.4}
	In 1990, Senn \cite{Senn} generated the Mandelbrot set using the hyperbolic numbers. Instead of obtaining a fractal structure, the set obtained seemed to be a square. Four years later, in \cite{MET}, Metzler proved that $\mHt$ is precisely a square with diagonal length 2$\frac{1}{4}$ and of side length $\frac{9}{8} \sqrt{2}$. It was also proved in \cite{RochonParise} that $\mHyb^3$ is a square with diagonal length $\frac{4}{3\sqrt{3}}$ and with side length $\frac{2}{3\sqrt{3}}\sqrt{2}$. In this subsection, we generalized their results for odd integers greater than two.
	
	Before proving the next theorem, we introduce some notations that the reader can find in \cite{RochonParise}. Let $(u,v)^{\top}, (x,y)^{\top} \in \mR^2$. We define two multiplication operations $\diamond$ and $\ast$ on $\mR^2$ as
	\begin{equation}\label{e3.1}
	\begin{pmatrix}
	u\\v
	\end{pmatrix} \diamond
	\begin{pmatrix}
	x\\y
	\end{pmatrix}:=
	\begin{pmatrix}
	ux+vy\\vx+uy
	\end{pmatrix}
	\end{equation}
	\begin{equation}\label{e3.2}
	\begin{pmatrix}
	u\\v
	\end{pmatrix} \ast
	\begin{pmatrix}
	x\\y
	\end{pmatrix}:=
	\begin{pmatrix}
	ux\\vy
	\end{pmatrix}\text{.}
	\end{equation}
It is easy to see that $(\mR^2 , + , \diamond )$, and $(\mR^2 , + , \ast )$ are commutative rings with zero divisors. Moreover, the multiplication $\diamond$ is exactly the same multiplication defined on $\mD$. So, in fact, we have $(\mR^2, +, \diamond ) \simeq (\mD, + , \cdot )$. If we set
	\begin{align*}
	T := \begin{pmatrix} 1 & -1 \\ 1 & 1 \end{pmatrix} 
	\end{align*}
then it is easy to see that $T$ is an isomorphism between $(\mR^2 , + , \diamond )$ and $(\mR^2 , +, \ast )$. Now, define
	\begin{align}
	\bHpc \begin{pmatrix}
	x \\ y
	\end{pmatrix}= \begin{pmatrix}
	x \\ y
	\end{pmatrix} \diamond_{\circ p} \begin{pmatrix}
	x \\ y
	\end{pmatrix} + \begin{pmatrix}
	a \\ b
	\end{pmatrix}
	\end{align}
where $\diamond_{\circ p}$ denotes the $p$-th composition of the operation $\diamond$ and $c = a + b \bj$. Then, we obtain the following property.
	\begin{lemma}\label{lemHyperIterationRealCompo}
	For all $m \in \mN$ such that $m \geq 1$, we have
		\begin{align*}
		T \bHpc^m \begin{pmatrix} x \\ y \end{pmatrix} = \begin{pmatrix}
		Q_{p,a - b}^m (x - y) \\ Q_{p,a + b}^m (x+y)
		\end{pmatrix} \text{.}
		\end{align*}
	\end{lemma}
	\begin{proof}
	We prove the statement by mathematical induction on $m$.
	For $m = 1$, we have
		\begin{align*}
		T \bHpc \begin{pmatrix} x \\ y \end{pmatrix} &= T \op \VEC{x}{y} \diamond_{\circ_{p}} \VEC{x}{y}  + \VEC{a}{b} \fp\\
		&= T \VEC{x}{y} \ast_{\circ_p} T \VEC{x}{y} + T \VEC{a}{b} \\
		&= \VEC{x - y}{x + y} \ast_{\circ_p} \VEC{x - y}{x+y} + \VEC{a - b}{a+b}\text{.}
		\end{align*}
	Moreover, from the definition of $\ast$, we obtain
		\begin{align*}
		T \bHpc \begin{pmatrix} x \\ y \end{pmatrix} &= \VEC{(x-y)^p}{(x+y)^p} + \VEC{a-b}{a+b} \\
		&= \VEC{Q_{p,a-b}(x-y)}{Q_{p,a+b}(x+y)} \text{.}
		\end{align*}
	For the general case, assume that the assumption is true for an integer $k \geq 1$. Then,
		\begin{align*}
		T \bHpc^{k+1} \begin{pmatrix} x \\ y \end{pmatrix} &= T \op \bHpc^{k} \VEC{x}{y} \diamond_{\circ_p} \bHpc^{k} \VEC{x}{y} + \VEC{a}{b} \fp \\
		&= T \bHpc^{k} \VEC{x}{y} \ast_{\circ_{p}} T \bHpc^{k} \VEC{x}{y} + T \VEC{a}{b} \text{.}
		\end{align*}
	From the induction hypothesis, we have
		\begin{align*}
		T \bHpc^{k+1} \begin{pmatrix} x \\ y \end{pmatrix} &= \VEC{Q_{p,a-b}^k(x-y)}{Q_{p,a+b}^k(x+y)} \ast_{\circ_{p}} \VEC{Q_{p,a-b}^k(x-y)}{Q_{p,a+b}^k (x + y)} + \VEC{a- b}{a+b}\\
		&= \VEC{(Q_{p,a-b}^k(x-y))^p}{(Q_{p,a+b}^k(x+y))^p} + \VEC{a-b}{a+b} \\
		&= \VEC{Q_{p,a-b}^{k+1}(x-y)}{Q_{p,a+b}^{k+1} (x+y)} \text{.}
		\end{align*}
		Hence, by the induction principle \begin{align*}
		T \bHpc^m \begin{pmatrix} x \\ y \end{pmatrix} = \begin{pmatrix}
		Q_{p,a - b}^m (x - y) \\ Q_{p,a + b}^m (x+y)
		\end{pmatrix}
		\end{align*}
		for all $m \in \mN^{*}$.

	\end{proof}
	
	We can now prove the following theorem.
	\begin{theorem}\label{t3.2}
		Let $p>2$ be an odd integer. Then, the Hyperbrot of order $p$ is characterized as
	\begin{equation*}
	\mHyb^p = \oa x+y\bj \in \mD \, : \, |x|+|y| \leq \frac{p-1}{p^{p/(p-1)}} \fa\text{.}
	\end{equation*}
	\end{theorem}
	\begin{proof}
	By Lemma 7 in \cite{RochonParise} and the remark right after, $\oa \bHpc^m (\mathbf{0}) \fa _{m=1}^{\infty}$ is bounded  iff the real sequences $\oa Q_{p,x-y}^m(0) \fa _{m=1}^{\infty}$ and $\oa Q_{p,x+y}^m(0) \fa _{m=1}^{\infty}$ are bounded. However, according to Theorem \ref{t2.3.1}, these sequences are bounded iff
		\begin{align}
		|x-y| \leq \frac{p-1}{p^{p/(p-1)}} \mbox{ and } |x+y|\leq \frac{p-1}{p^{p/(p-1)}} \label{e3.7}\text{.}
		\end{align}
	Then, by a simple computation we obtain that $|x|+|y| \leq \frac{p-1}{p^{p/(p-1)}}$. Conversely, if $|x|+|y| \leq \frac{p-1}{p^{p/(p-1)}}$ is true, then by the properties of the absolute value, we obtain directly the inequalities in \eqref{e3.7}. Thus, we obtain the desired characterization for $\mHyb^p$.
	\end{proof}
	
We see that the images in figure \ref{figureHyperbrots} represent faithfully Theorem \ref{t3.2}. Moreover, we remark that the squares are growing as $p$ tends to infinity. Does this process stop or goes \textit{ad infinitum}?
	
\begin{figure}
\centering
\subfigure[$\mHyb^3$]{
	\includegraphics[scale = 0.068]{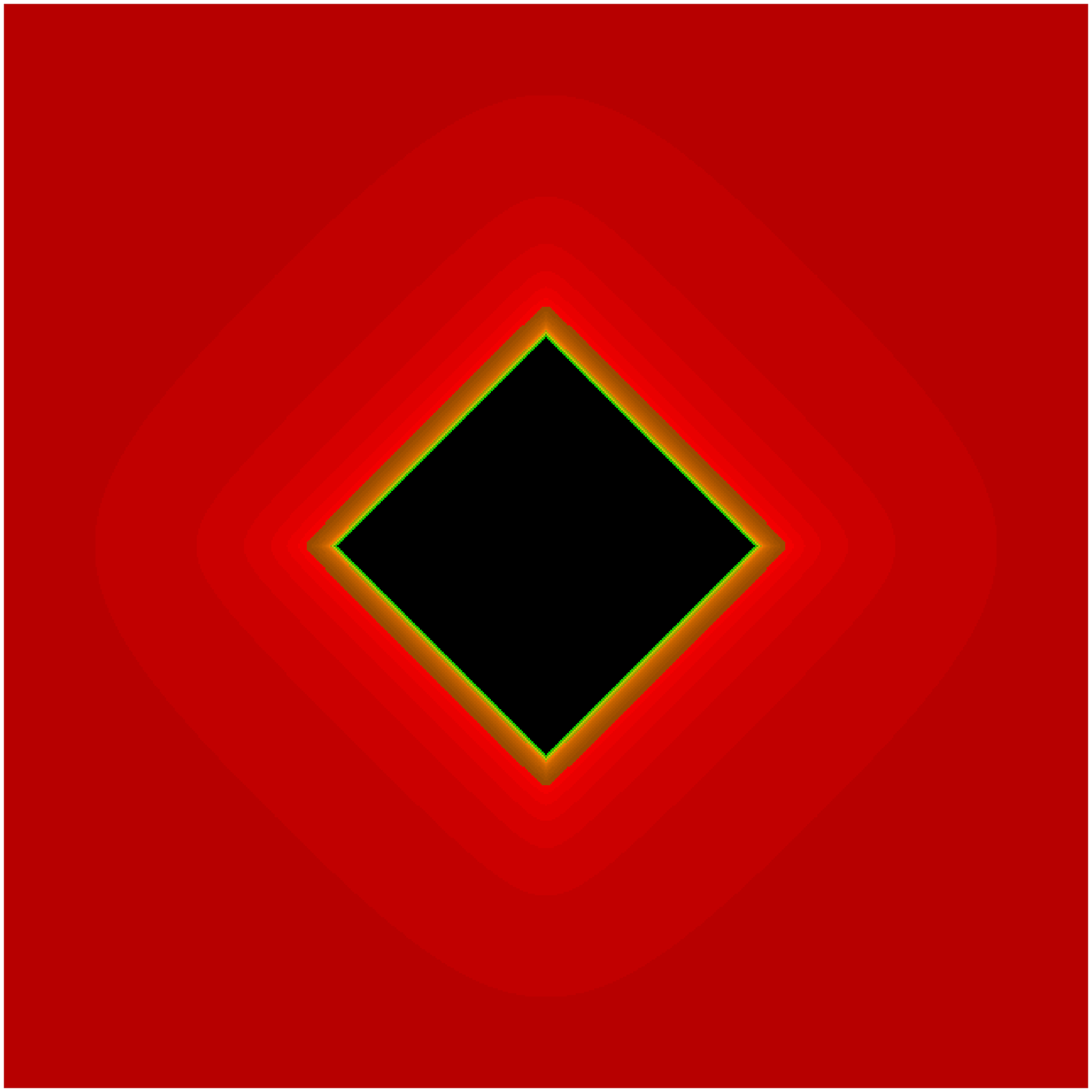}}
\subfigure[$\mHyb^5$]{
	\includegraphics[scale = 0.068]{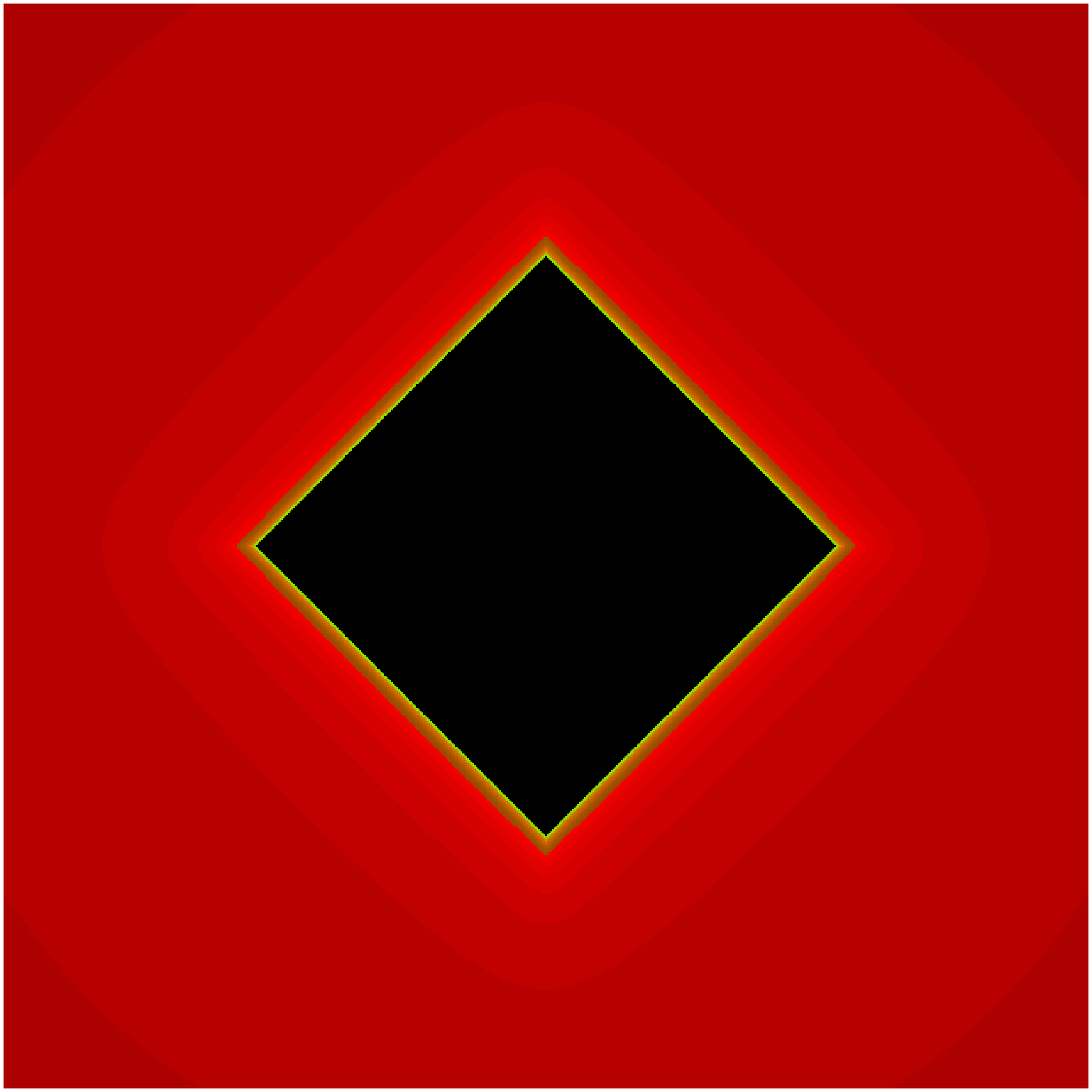}}
\subfigure[$\mHyb^7$]{
	\includegraphics[scale = 0.068]{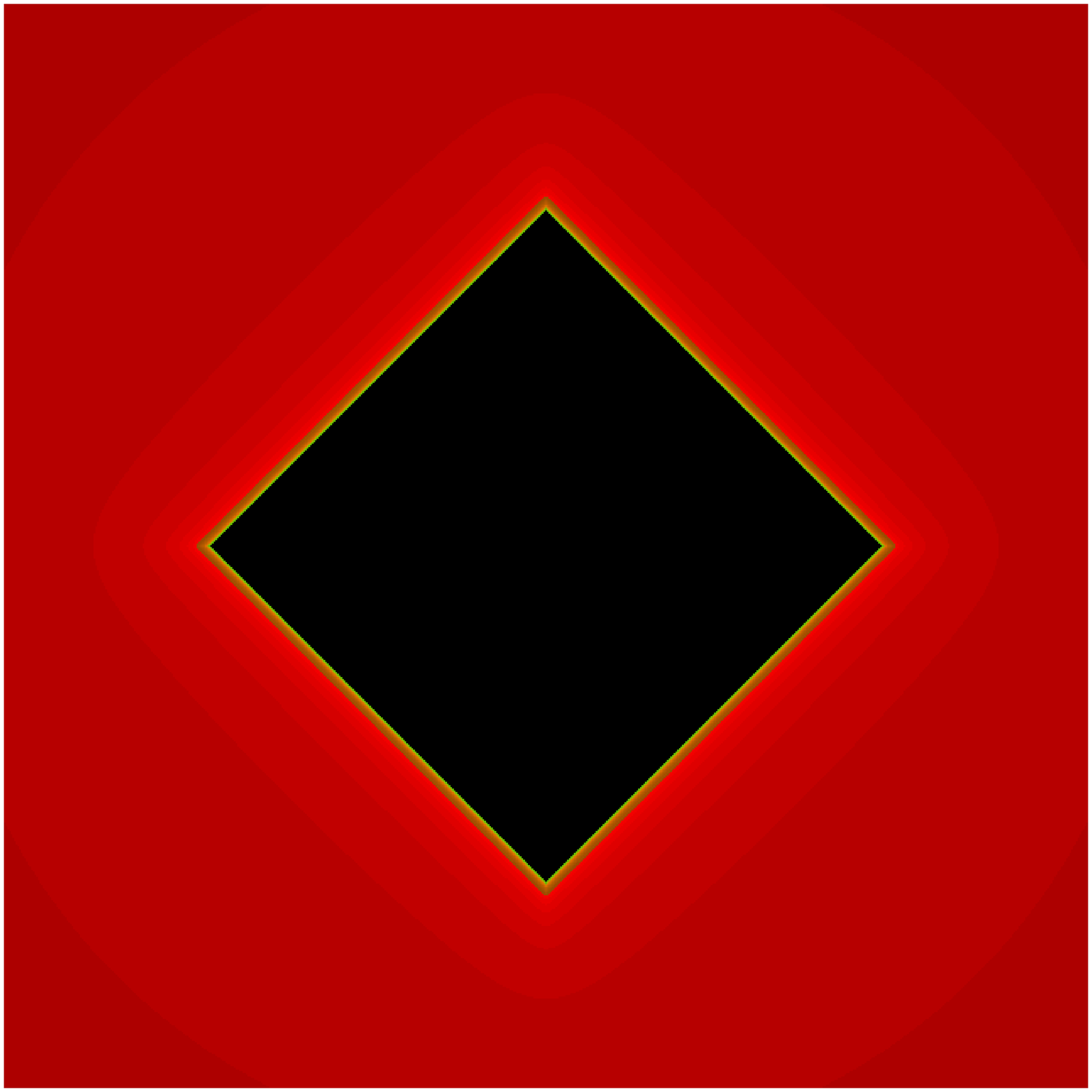}}
\subfigure[$\mHyb^9$]{
	\includegraphics[scale = 0.068]{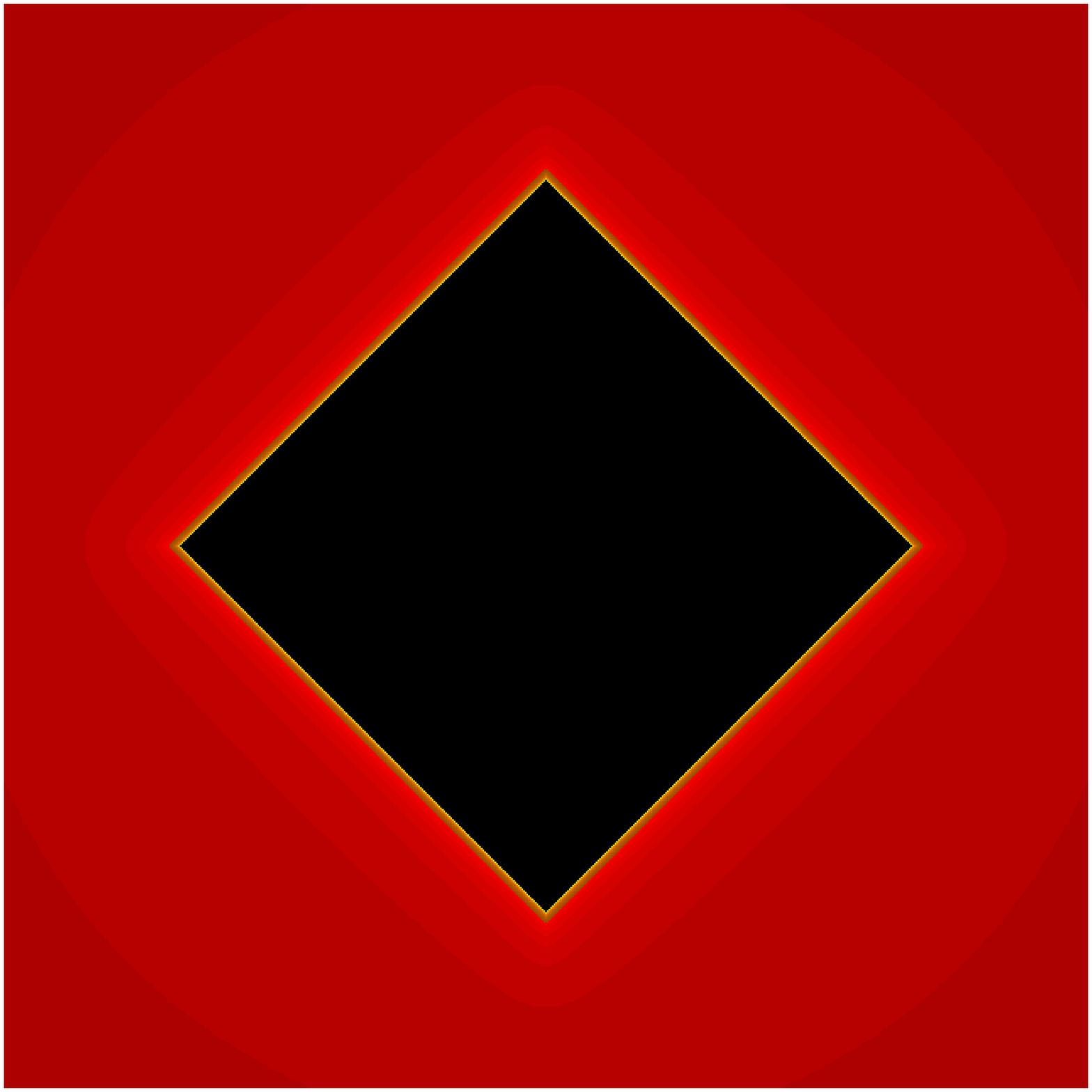}}
\subfigure[$\mHyb^{11}$]{
	\includegraphics[scale = 0.068]{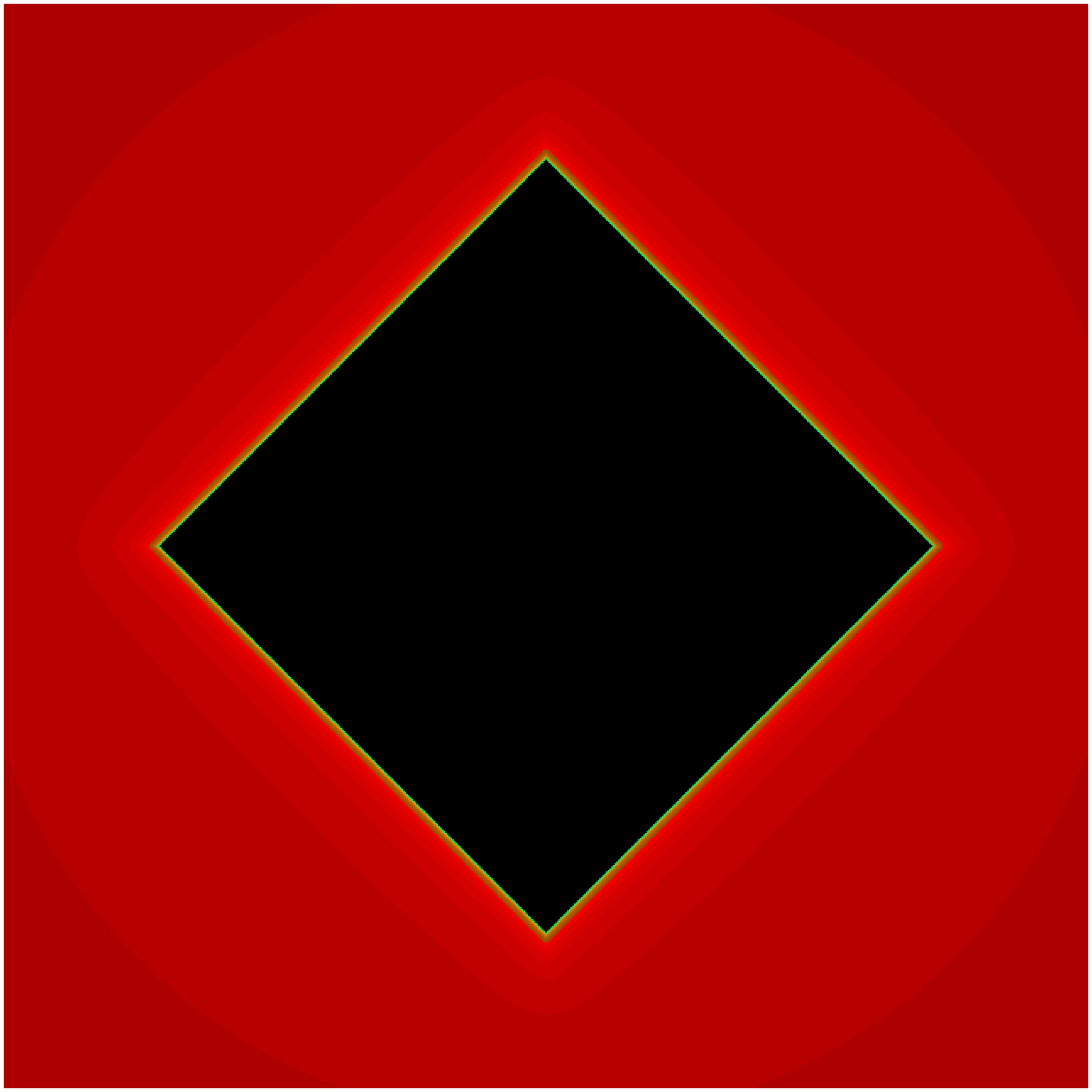}}
\subfigure[$\mHyb^{13}$]{
	\includegraphics[scale = 0.068]{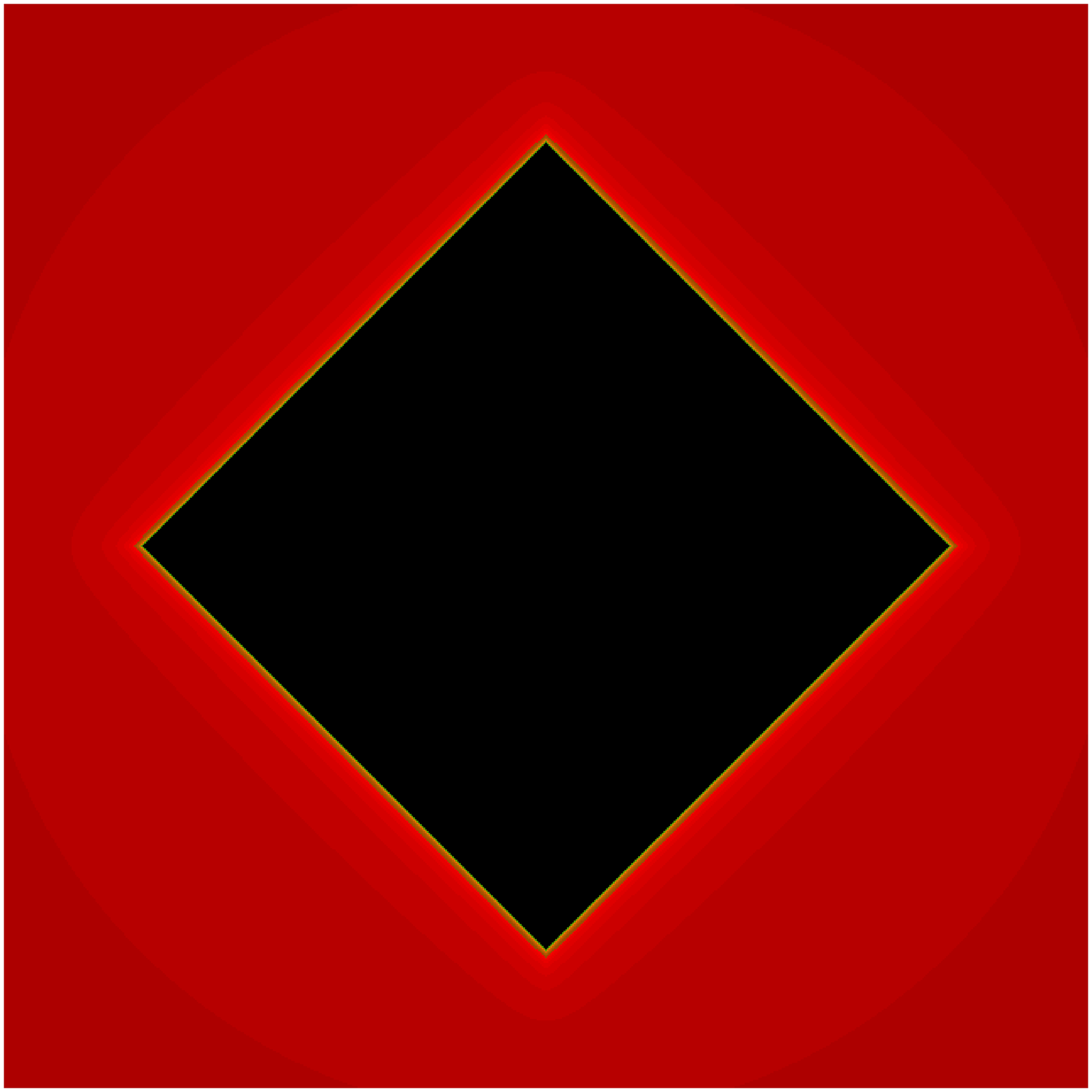}}
\caption{Hyperbrots for several odd integers, $-1 \leq \pre (c),  \mathrm{Hy} (c) \leq 1$. The black color represents the Hyperbrot.}\label{figureHyperbrots}
\end{figure}

To answer this question, let $\mathcal{S}(\mR^2)$ be the collection of non-empty compact subsets of $\mR^2$ and define the distance between $A$ and $B$ in $\mathcal{S}(\mR^2)$ as
	\begin{align*}
	d(A,B)&:= \max_{x \in A} \oa d(x,B) \fa \\
	&= \max_{x \in A} \oa \min_{y \in B} \oa \Vert x - y \Vert \fa \fa
	\end{align*}
where $\Vert \cdot \Vert$ is the Euclidean distance on $\mR^2$. With these definitions from \cite{Barnsley}, we have the following theorem.
	\begin{theorem}\label{HausdorffH2D}
	Let $(\mathcal{S}(\mR^2), h)$ be the so-called \textbf{Fractals metric space} on $\mR^2$ where $\mathcal{S}(\mR^2)$ is the collection of all non-empty compact subsets of $\mR^2$ and $h: \, \mathcal{S}(\mR^2) \times \mathcal{S}(\mR^2) \ra [0,+\infty )$ is the Hausdorff distance on the collection $\mathcal{S}(\mR^2)$ defined as $h(A,B) = \max \oa d(A,B) , d(B,A) \fa$. Let $\mHyb := \oa x+y\bj\in\mD \, : \, |x| + |y| \leq 1 \fa$. Then,
		\begin{align*}
		\lim_{n \ra \infty} h(\mHyb , \mHyb^{2n+1}) = 0 \text{.}
		\end{align*}
	\end{theorem}
	
	\begin{proof}
	Let $p > 2$ be an odd integer and let $m_p := \frac{p-1}{p^{p/(p-1)}}$. From Theorem \ref{t3.2}, we see that $\mHyb^p \subset \mHyb$ since $m_p < 1$ $\forall p > 2$. Then, $d(\mHyb^p , \mHyb ) = 0$. So, we just have to calculate $d(\mHyb , \mHyb^p )$. 
	
	Suppose that $c=a+b\bj\in\mD$. Let $z=x+y\bj \in \mHyb^p$ and define the function $f(\cdot , c): \mHyb^p \ra [0,\infty )$ by
		\begin{align*}
		f(z,c) := \Vert c - z \Vert^2\text{.}
		\end{align*}			

	Obviously, if $c \in \mHyb^p$, then $f(c,c) = 0$ is the minimum on $\mHyb^p$. Moreover, if $c \in \mHyb - \mHyb^p$, the minimum must be on $\partial \mHyb^p= \oa x+y\bj\in\mD \, : \, |x| + |y| = m_p \fa$. Then, let define the function $g: \mHyb - \mHyb^p \ra [0,+\infty )$ by
	
	\begin{align*}
	g(c):=& \min_{z \in \partial \mHyb^p} \oa f(z,c) \fa \text{.}
	\end{align*}
	
Geometrically, the maximum of $g(c) $ occurs at the points $c_1 = (1,0)$, $c_2 = (0,1)$, $c_3 = (-1,0)$ and $c_4 = (0,-1)$ and the values are $g_{max}:= g(c_k) = (1-m_p)^2$ for $k = 1,2,3,4$. Thus, $h(\mHyb, \mHyb^p) = \sqrt{g_{max}} = (1-m_p)$. Finally, since $\lim_{n \ra \infty} m_{2n+1} = 1$, then we obtain
	\begin{align*}
	\lim\limits_{n\ra \infty} h(\mHyb , \mHyb^{2n+1}) = 0 \text{.}
	\end{align*}
	\end{proof}
	
	\subsection{Characterization of the generalized Perplexbrot}
	In this subsection, we generalize Hyperbrots in three dimensions. Let's adopt the same notation as in \cite{RochonParise} and \cite{GarantRochon} for the generalized \textit{Perplexbrot}
	\begin{align}
	\mathcal{P}^p:=\mathcal{T}^p(1,\bjp,\bjd)
=\oa c=c_0+c_5\bjp+c_6\bjd \, : \, c_i \in \mathbb{R} \text{ and } \right. \notag\\
\left. \oa Q_{p,c}^m(0)\fa_{m=1}^{\infty} \text{ is bounded} \fa \text{.}\label{eq3.3.1}
	\end{align}
	We say that $\mathcal{P}^p$ is the generalized Perplexbrot since in \cite{GarantRochon}, the classical Perplexbrot is $\mathcal{P}^2$. We note that a generalized Perplexbrot of order $p$ is in fact a specific 3D slice of a tricomplex Multibrot of order $p$ as for the Tetrabric, Hourglassbric and Metabric introduced in \cite{RochonParise} for the case $p=3$. Before proving the fact that $\mathcal{P}^p$ is an octahedron for all odd integers $p>2$, we need this following lemma.

	\begin{lemma}\label{lem3.3.1}
	We have the following characterization of the generalized Perplexbrot
		\begin{equation*}
		\mathcal{P}^p=\bigcup_{y\in \left[-m_p,m_p\right]} \oa \left[ (\mathcal{H}^p-y\bjp)\cap (\mathcal{H}^p+y\bjp)\right] + y\bjd\fa
		\end{equation*}
	where $\mathcal{H}^p$ is the Hyperbrot for an odd integer $p > 2$ and $m_p:=\frac{(p-1)}{p^{p/(p-1)}}$.
	\end{lemma}

	\begin{proof}
	By Definition of $\mathcal{P}^p$ and the idempotent representation, we have that
		\begin{equation}
		\mathcal{P}^p=\oa c=\left( d-c_6\bjp\right) \gamma_2 + \left( d + c_6\bjp\right) \overline{\gamma}_2 \, : \, \oa Q_{p,c}^m(0)\fa_{m=1}^{\infty} \text{ is bounded} \fa \label{PDefRemod}
		\end{equation}
	where $d=c_0+c_5\bjp \in \mH (\bjp )$. Furthermore, the sequence $\oa Q_{p,c}^m(0)\fa_{m=1}^{\infty}$ is bounded iff the two sequences $\oa Q_{p,d-c_6\bjp}^m(0)\fa_{m=1}^{\infty}$ and $\oa Q_{p,d+c_6\bjp}^m(0)\fa_{m=1}^{\infty}$ are bounded. To continue, we make the following remark about hyperbolic dynamics: $\forall z \in \mathbb{D}$
		\begin{equation}
		\mathcal{H}^p-z:=\oa c \in \mathbb{D} \, | \, \oa Q_{p,c+z}^m(0)\fa_{m=1}^{\infty} \text{ is bounded } \fa \text{.}\label{eq3.3.16}
		\end{equation}
	By Definition \ref{d3.1}, $\oa Q_{p,d-c_6\bjp}^m(0)\fa_{m=1}^{\infty}$ and $\oa Q_{p,d+c_6\bjp}^m(0)\fa_{m=1}^{\infty}$ are bounded iff $d-c_6\bjp, d+c_6\bjp \in \mathcal{H}^p$. Therefore, by \eqref{eq3.3.16}, we also have that $d-c_6\bjp, d+c_6\bjp \in \mathcal{H}^p$ iff $d \in (\mathcal{H}^p-c_6\bjp) \cap (\mathcal{H}^p+c_6\bjp)$. Hence,
		\begin{align*}
		\mathcal{P}^p&=\oa c=c_0+c_5\bjp+c_6\bjd \, | \, c_0+c_5\bjp \in (\mathcal{H}^p-c_6\bjp) \cap (\mathcal{H}^p+c_6\bjp) \fa \\
		&=\bigcup_{y\in \mathbb{R}} \oa \left[(\mathcal{H}^p-y\bjp)\cap (\mathcal{H}^p+y\bjp)\right] + y\bjd\fa\text{.}
		\end{align*}
	In fact, by Theorem \ref{t3.2},
		\begin{equation}
		(\mathcal{H}^p-y\bjp)\cap (\mathcal{H}^p+y\bjp)=\emptyset
		\end{equation}
	whenever $y\in \left[ -\frac{p-1}{p^{p/(p-1)}},\frac{p-1}{p^{p/(p-1)}}\right]^{c}$. This conduct us to the desire characterization of the generalized Perplexbrot.
	\end{proof}

We remark that for a given $y \in [- m_p , m_p ]$, the intersection $(\mathcal{H}^p-y\bjp)\cap (\mathcal{H}^p+y\bjp)$ is exactly the cross section parallel to the $xy$-plane in $\mR^3$ of an octahedron of diagonal length equals to $2m_p$. Hence, as a consequence of Lemma \ref{lem3.3.1} and Theorem \ref{t3.2}, we have the following corollary illustrated by figure \ref{figureHyperbrots2}:

	\begin{corollary}\label{theo3.2.5}
	$\mathcal{P}^p$ is a regular octahedron of edges $\sqrt{2}\frac{(p-1)}{p^{p/(p-1)}}$ where $p>2$ is an odd integer. Moreover, the generalized Perplexbrot can be rewritten as the set
		\begin{equation}
		\mathcal{P}^p = \oa x + y\bjp + z \bjd \, : \, (x,y,z) \in \mR^3 \, \text{ and } \, |x| + |y| + |z| \leq m_p \fa\text{.}
		\end{equation}
	\end{corollary}
	
\begin{figure}
\centering
\subfigure[$\cP^3$]{
	\includegraphics[scale = 0.2]{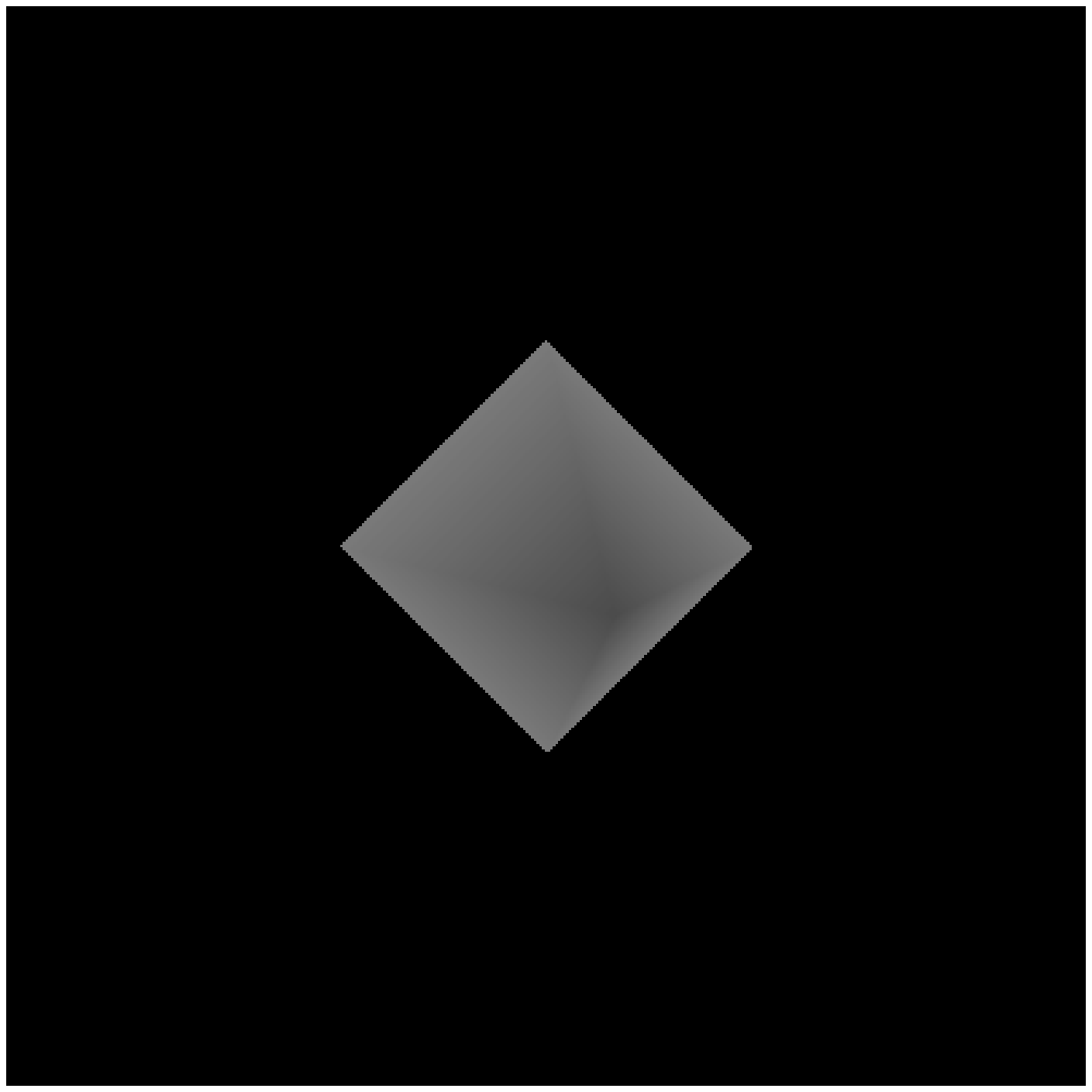}}
\subfigure[$\cP^5$]{
	\includegraphics[scale = 0.2]{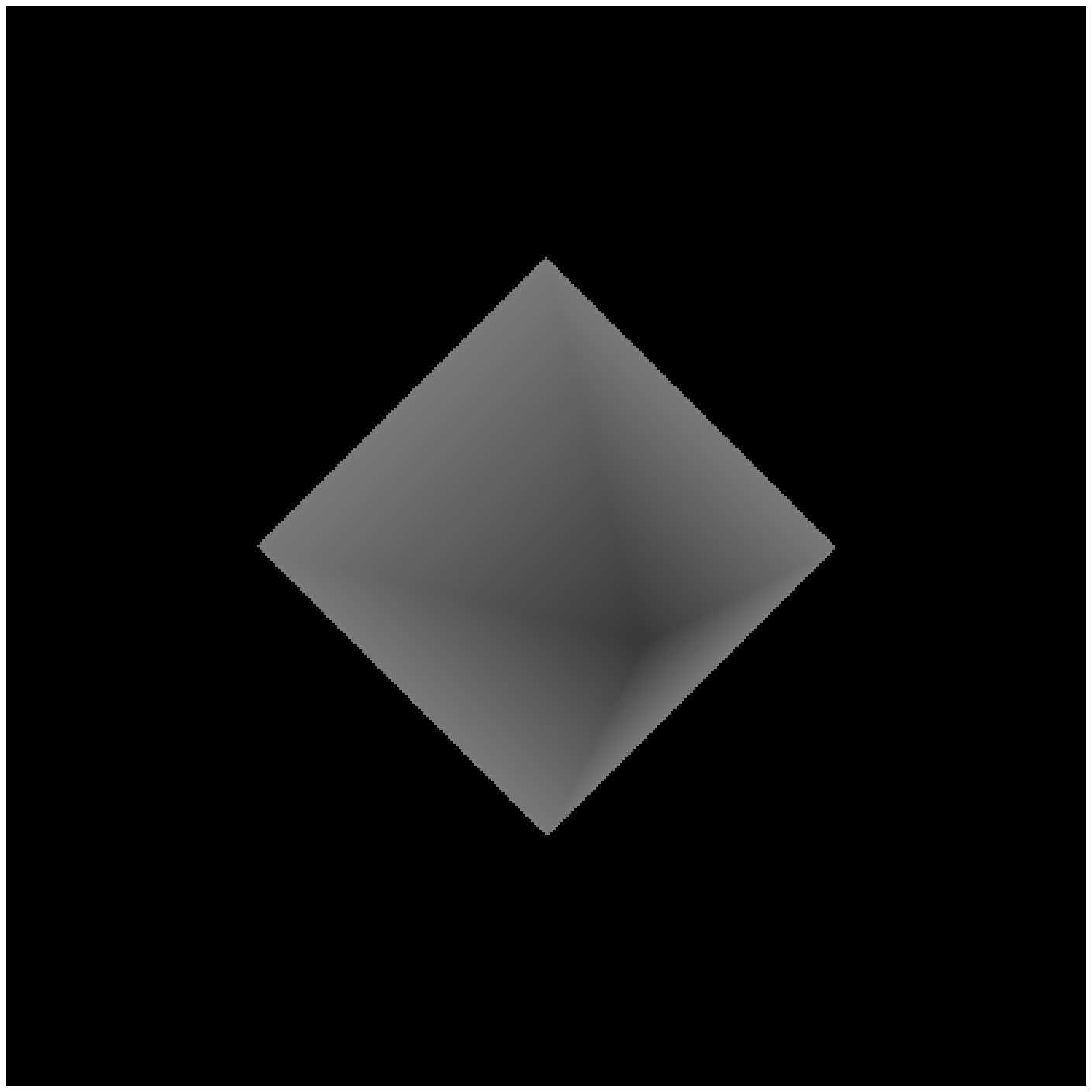}}
\subfigure[$\cP^9$]{
	\includegraphics[scale = 0.2]{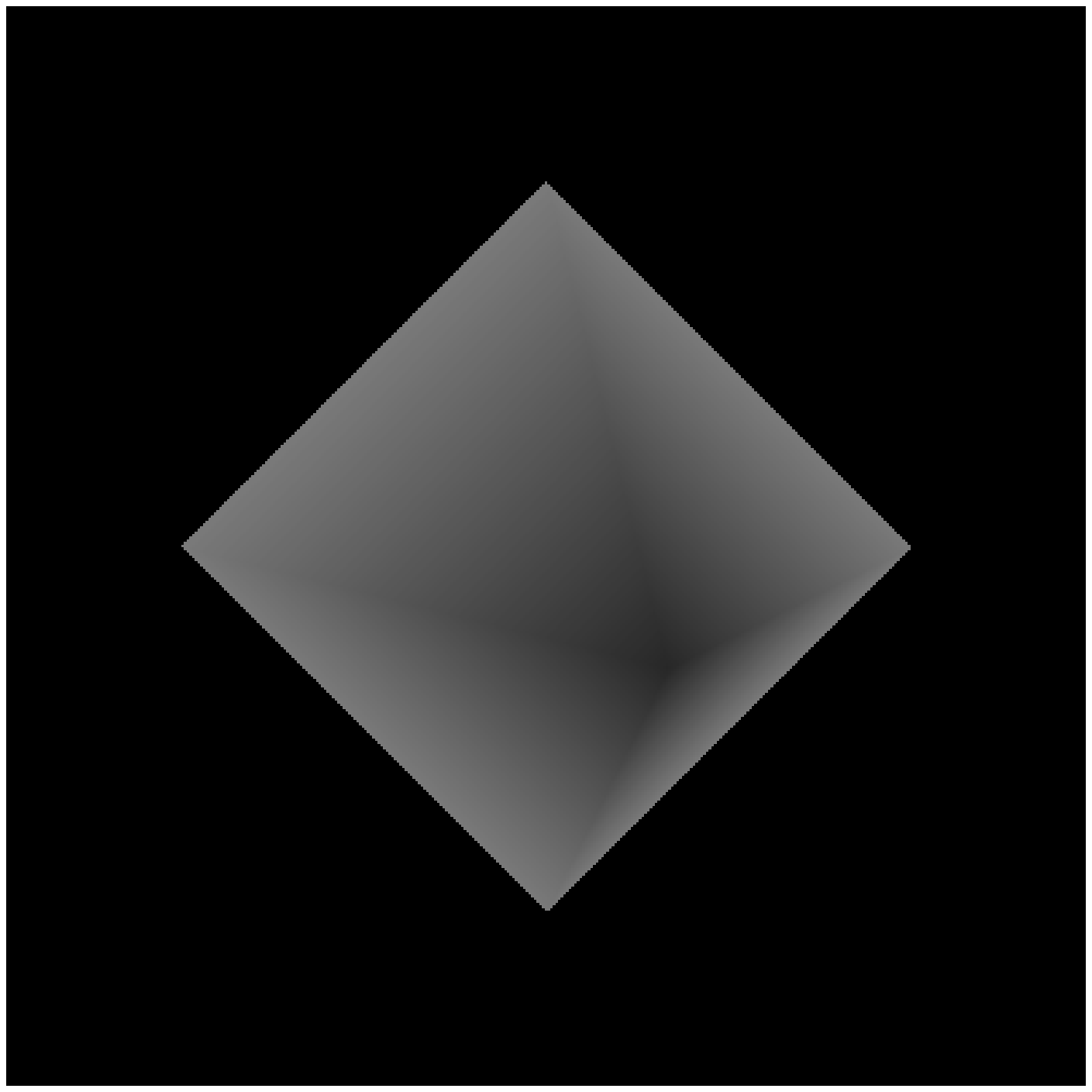}}
\subfigure[$\cP^{13}$]{
	\includegraphics[scale = 0.2]{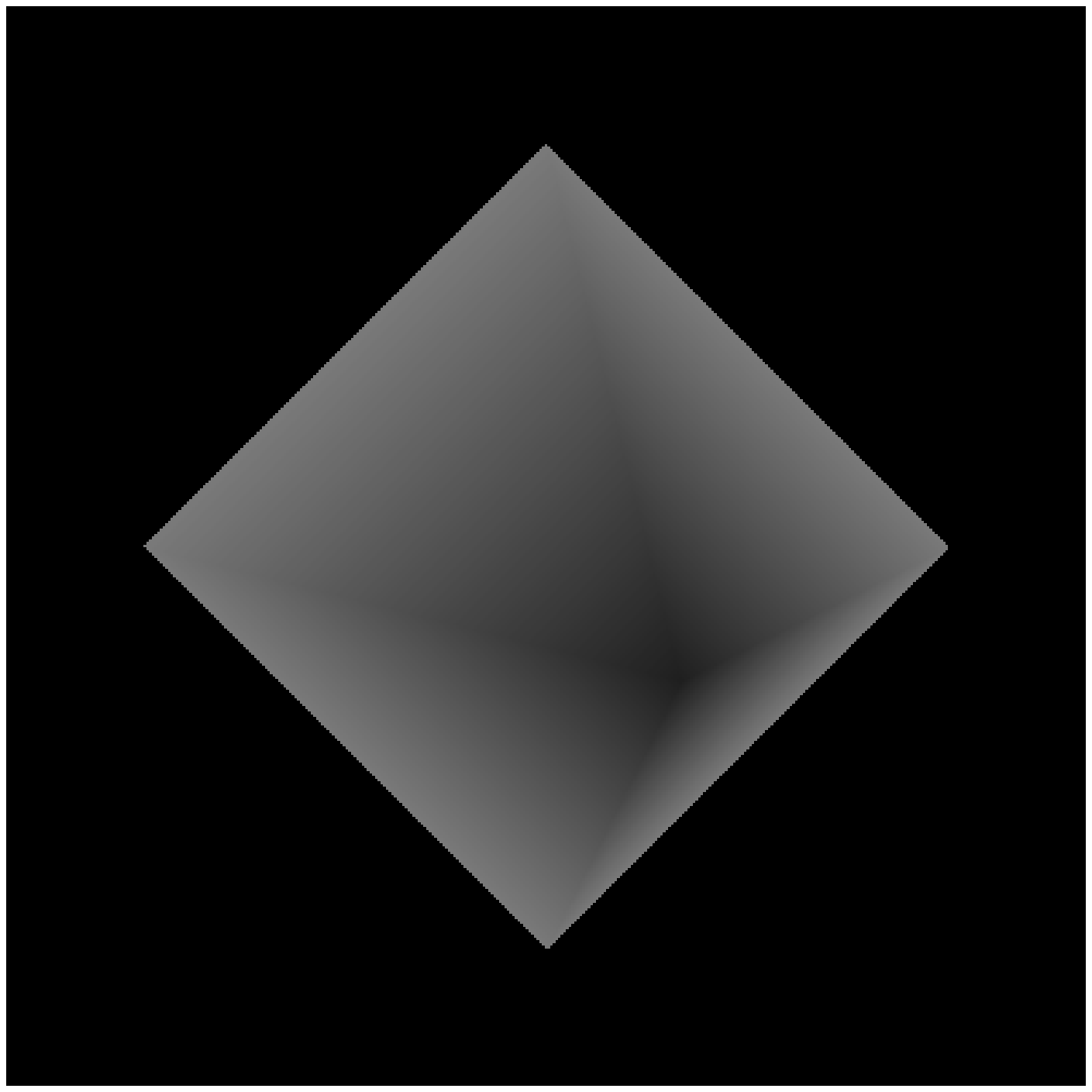}}
\caption{Hyperbrots for several odd integers, $-1 \leq x,y,z \leq  1$}\label{figureHyperbrots2}
\end{figure}

Similarly to what was done for Hyperbrot sets, we prove that as $ p \ra \infty$, the sequence of generalized Perplexbrots converges to a non-empty compact subset of $\mR^3$.

	\begin{theorem}
	Let $(\mathcal{S}(\mR^3), h)$ be the so-called \textbf{Fractals metric space} on $\mR^3$ where $\mathcal{S}(\mR^3)$ is the collection of all non-empty compact subsets of $\mR^3$ and $h: \, \mathcal{S}(\mR^3) \times \mathcal{S}(\mR^3) \ra [0,+\infty )$ is the Hausdorff distance on the collection $\mathcal{S}(\mR^3)$ defined as $h(A,B) = \max \oa d(A,B) , d(B,A) \fa$. Let
	\begin{align*}
		\cP := \oa x + y \bjp + z\bjd \, : \,(x,y,z)\in \mR^3\, \text{ and } \, |x| + |y| + |z| \leq 1 \fa\text{.}
	\end{align*}
	Then, 
		\begin{align*}
		\lim_{n \ra \infty} h(\cP , \cP^{2n+1}) = 0 \text{.}
		\end{align*}
	\end{theorem}
	
	\begin{proof}
	Let $p > 2$ be an odd integer and let $m_p = \frac{p-1}{p^{p/(p-1)}}$. Using elementary calculus and by referring the reader to the proof of the Theorem \ref{HausdorffH2D} for the two dimensional case, we find that $h(\cP , \cP^p) = (1-m_p)$ where the maximum occurs at the corners of the set $\cP$. Hence, 
		\begin{align*}
		\lim_{n \ra \infty} h(\cP , \cP^{2n+1}) = 1 - 1 = 0
		\end{align*}		 
	since $\lim\limits_{n\ra \infty} m_{2n+1} = 1$.
	\end{proof}
	
\section*{Conclusion}
In this work, we treated Multibrot sets for a specific polynomial of odd degree. We saw that complex and hyperbolic dynamical systems can be treated entirely with the tricomplex space. We also saw that complex and hyperbolic dynamical systems generated by the polynomial $Q_{p,c}$ for odd integers are really different: the first one gives irregular shapes and the second one gives regular shapes as squares. Moreover, the generalize 3D versions of the Hyperbrot sets are regular octahedron.

For the case of complex \textit{Multibrot} sets, it would be grateful if we can grade-up the proof of Theorem \ref{t2.3.1} for all Multibrots of even powers. For this, we need to find a specific approach to prove the following conjecture.
	\begin{conjecture}\label{cj1}
	Let $\mManp$ be the generalized Mandelbrot set for the polynomial $Q_{p,c}(z)=z^p+c$ where $z,c \in \mC$ and $p\geq 2$ an integer. If $p$ is even, then
		\begin{equation}
		\mManp \cap \mR = \left[-2^{\frac{1}{p-1}},(p-1)p^{\frac{-p}{p-1}} \right]
		\end{equation}
	\end{conjecture}

This result would conducts us to another conjecture about the Hyperbrots.
	\begin{conjecture}\label{cj2}
	The Hyperbrots are squares and the following characterization of \textit{Hyperbrots} for even powers $p \geq 2$ holds:
		\begin{equation}
		\mHyb^p = \oa x+y\bj\in\mD \, : \, |x-t_p| + |y| \leq \frac{l_p}{2} \fa
		\end{equation}
	where 
		\begin{align*}
		t_p:= \frac{(1 - (2p)^{1/(p-1)})p - 1}{2p(p^{1/(p-1)})} \quad \text{ and } \quad l_p:= \frac{((2p)^{1/(p-1)} + 1)p - 1}{p(p^{1/(p-1)})}\text{.}
		\end{align*}
	\end{conjecture}
	
Further explorations of 3D slices of the tricomplex Multibrot sets are also planned.

\section*{Acknowledgement}
DR is grateful to the Natural Sciences and
Engineering Research Council of Canada (NSERC) for financial
support. POP would also like to thank the NSERC for the award of a Summer undergraduate Research grant.

%-----------Bibliograhy-------------------------------------------------------------------------------

\end{document}